\documentclass{amsart}
\usepackage[left=1in, right=1in, top=1in, bottom=1in]{geometry}

\usepackage[utf8]{inputenc}
\usepackage{amssymb}
\usepackage{mathtools}
\usepackage{derivative}
\usepackage{mathrsfs}
\usepackage{array, ltablex, multirow} % tabularx: auf Textbreite vergrößern
\usepackage{makecell}
\usepackage{amsmath}
\usepackage{amsthm}
\usepackage{url}
\usepackage{tikz-cd}
\usepackage{graphicx}
\usepackage{tabularx}
\usepackage{hyperref}% http://ctan.org/pkg/hyperref
\hypersetup{
    colorlinks=true,
    linkcolor=blue,
    filecolor=magenta,      
    urlcolor=cyan,
    pdftitle={Overleaf Example},
    pdfpagemode=FullScreen,
    }

\usepackage[T1]{fontenc}
\usepackage[utf8]{inputenc}
\usepackage{babel}
\usepackage{amsthm}
\usepackage{tikz}
\usepackage{xparse}
\let\realItem\item % save a copy of the original item
\makeatletter
\NewDocumentCommand\myItem{ o }{%
   \IfNoValueTF{#1}%
      {\realItem}% add an item
      {\realItem[#1]\def\@currentlabel{#1}}% add an item and update label
}
\makeatother
\usepackage{enumitem}
\setlist[enumerate]{
    before=\let\item\myItem,       % use \myItem in enumerate
    label=\textnormal{(\arabic*)}, % format the label
    widest=(2')                    % set the widest label
}

\usepackage{aliascnt}

\title{Stable Quadratic Polynomials over $\Qi$}
\author{Jermain McDermott}

\theoremstyle{plain}
\newtheorem{theorem}{Theorem}[section]
\newtheorem{lemma}[theorem]{Lemma}
\newtheorem{corollary}[theorem]{Corollary}
\newtheorem{proposition}[theorem]{Proposition}
\newtheorem{conjecture}[theorem]{Conjecture}
\newtheorem{definition}[theorem]{Definition}

\usepackage[noabbrev]{cleveref}

\newtheorem*{remark}{Remark}
\usepackage{cleveref}
\crefname{theorem}{Theorem}{Theorems}
\crefname{lemma}{Lemma}{Lemmas}
\crefname{corollary}{Corollary}{Corollaries}
\crefname{proposition}{Proposition}{Propositions}
\crefname{definition}{Definition}{Definitions}
\crefname{conjecture}{Conjecture}{Conjectures}

\pagecolor{white}
\color{black}

\begin{document}

\newcommand{\U}{\mathscr{U}}
\newcommand{\Q}{\mathbb{Q}}
\newcommand{\Qi}{\mathbb{Q}(i)}
\newcommand{\Z}{\mathbb{Z}}
\newcommand{\Zi}{\mathbb{Z}[i]}
\newcommand{\R}{\mathbb{R}}
\newcommand{\C}{\mathbb{C}}
\newcommand{\K}{\overline{K}}
\newcommand{\F}{\mathbb{F}_p}
\newcommand{\Qx}{\mathbb{Q}[x]}
\newcommand{\Zx}{\mathbb{Z}[x]}
\newcommand{\Fx}{\mathbb{F}_p[x]}
\newcommand{\Cl}{\text{Cl}}
\newcommand{\Zm}{(\Z/m\Z)^\times}
\newcommand{\Zq}{(\Z/q\Z)^\times}
\newcommand{\Zp}{(\Z/p\Z)^\times}
\newcommand{\Gal}{\text{Gal}}
\newcommand{\I}{\langle i \rangle}
\newcommand{\ord}{\text{ord}}
\newcommand{\Real}{\text{Re}}
\newcommand{\Imag}{\text{Im}}

\newcommand{\genlegendre}[4]{%
  \genfrac{(}{)}{}{#1}{#3}{#4}%
  \if\relax\detokenize{#2}\relax\else_{\!#2}\fi
}
\newcommand{\legendre}[3][]{\genlegendre{}{#1}{#2}{#3}}
\newcommand{\dlegendre}[3][]{\genlegendre{0}{#1}{#2}{#3}}
\newcommand{\tlegendre}[3][]{\genlegendre{1}{#1}{#2}{#3}}

\newcommand{\genglegendre}[4]{%
  \genfrac{[}{]}{}{#1}{#3}{#4}%
  \if\relax\detokenize{#2}\relax\else_{\!#2}\fi
}
\newcommand{\glegendre}[3][]{\genglegendre{}{#1}{#2}{#3}}
\newcommand{\dglegendre}[3][]{\genglegendre{0}{#1}{#2}{#3}}
\newcommand{\tglegendre}[3][]{\genglegendre{1}{#1}{#2}{#3}}

\begin{abstract}

 %Let $f$ be a polynomial or a rational function over a field $K$. A basic question is, if $f$ is a polynomial, are its iterates irreducible or not? 
 We study iterates of a quadratic $f= x^2+1/c\in K[x]$. If the number of factors of $f^n:=f\circ f \circ ... \circ f$ is bounded by a constant independent of $n$, then $f$ is said to be \emph{eventually stable}. This paper is an extension to $\Qi$ of the paper \cite{evstb},  which considered $f$ over $\mathbb{Q}$. The conjecture "if $f^2$ is irreducible, then $f^n$ is irreducible for all $n$" extends to $\Qi$, but the lack of a linear ordering on $\Qi$ requires us to pass to the norm. As a result, an auxiliary function is necessary for a lower bound on a finite $n$ to check. The elusive case of $c\equiv 2 \bmod 4$ (as a $\Z$ equivalence class) is shown to be "stable" over $\Qi$, offering more evidence for \cite[Conjecture 1]{evstb}. Stability for $c\equiv 1\bmod 2$ (as a $\Zi$ equivalence class) is not as fully handled as over $\Z$, however. \\

\end{abstract}

\maketitle

\section{Introduction}
\quad Suppose $K$ is a field with algebraic closure $\K$. Let $f(x)=x^2+r\in K[x]$, and denote its iterates by  $$f^n(x):=(f\circ f \circ ... \circ f)(x) \text{  ($n$ compositions) } \quad \text{ with  } f^0(x)=x.$$ 

%In this paper we study the arithmetic dynamics of quadratic polynomials $f(x)=x^2+r\in K[x]$ under iteration, i.e. $$f^n(x)=(f\circ f \circ ... \circ f)(x) \text{  ($n$ compositions) } \quad \text{ with  } f^0(x)=x.$$ 

Fix $\alpha\in K$ and denote by $f^{-n}(\alpha):=\{\beta\in \K: f^n(\beta)=\alpha\}.$ If $f^n(x)-\alpha$ is separable, the disjoint union $T_{f,\infty}(\alpha):=\{\alpha\}\sqcup f^{-1}(\alpha) \sqcup f^{-2}(\alpha) \sqcup \cdots $ is a rooted tree with root vertex $\alpha$, and edges assigned according to the action of $f$. Denote by $\text{Aut}(T_f(\alpha))$ the tree automorphisms of $T_f(\alpha)$. Since the natural action of the Galois group $G_K:=\Gal(\K/K)$ on $f^{-n}(\alpha)$ commutes with $f\in K[x]$, we obtain a homomorphism $\Gal(\K/K)\to \text{Aut}(T_f(\alpha))$ called the arboreal Galois representation associated to $(f,\alpha)$. One of the central problems in arithmetic dynamics is whether the image of this homomorphism has finite index in $\text{Aut}(T_f(\alpha))$; a dynamical parallel of the Serre open image theorem. This question has generated a large literature, see \cite{currenttrends} for a summary of this and other current research in the field. \\

In this article, we attempt to characterize factorizations of $f^n(x)-\alpha$ for $\alpha=0$.

%Definition 1
\begin{definition}
Let $K$ be a field,  $f\in K[x]$, and $\alpha\in K$. For $n\ge 1$ let $k_n$ denote the number of irreducible factors of $f^n(x)-\alpha$ over $K$. We call the pair $(f,\alpha)$ (resp. $f$) $\textbf{eventually stable over $K$}$ if there is a constant $C(f,\alpha)$ (resp. $C(f, 0)$) such that for all $n$, $$k_n\le C(f,\alpha).$$
The pair $(f,\alpha)$ is called \textbf{stable over $K$} if $C(f,\alpha)=1$, i.e. $f^n(x)-\alpha$ is irreducible over $K$ for all $n\ge 1$.
%We call $f$ is eventually stable over K if the pair (f,0) is eventually stable over K.\\

%We call the pair $(f,\alpha)$ (resp. $f$) \textbf{stable over K} if $C(f,\alpha)=1$ (resp. $C(f,0)=1$), i.e. when $f^n(x)-\alpha$ (resp. $f^n$) is irreducible over $K$ for all $n$.
\end{definition}

%A conjecture of Jones and Levy posits that if $f(x)\in K(x)$ is a rational function and $\alpha$ is not periodic under $f$, then $(f,\alpha)$ is eventually stable.\\

Eventual stability is applied in some finite-index results for certain arboreal Galois representations \cite{18finite,17finite}, and is equivalent to a bound independent of $n$ on the number of Galois orbits on $f^{-n}(\alpha)$ \cite[Proposition 2.2]{ratstb}. However, this is a straightforward consequence of the image of the arboreal representation associated to $(f,\alpha)$ having finite index, so finite index of the arboreal representation implies eventual stability. The reverse implication is much less clear. Other applications include finiteness of $S$-integer points in backwards orbits; see \cite[Section 3]{ratstb} for a discussion of these and other results concerning the topic. \\

The following theorem of Hamblen, Jones, and Madhu \cite{orbits} establishes eventual stability of a large class of polynomials over number fields (using ideas resembling Eisenstein's criterion):
\begin{theorem}
\label[theorem]{evstbthm} (Hamblen, Jones and Madhu, \cite[Theorem 5]{orbits})
Let $d\ge 2$, let $K$ be a field of characteristic not dividing $d$, and let $f(x)=x^d+r\in K[x]$. If there is a discrete non-archimedean absolute value on $K$ with $|r|<1$, then $f$ is eventually stable over $K$.
\end{theorem}

The following corollary is an extension of \cite[Corollary 6]{orbits} considered by DeMark et al.:
\begin{corollary}
    Let $f(x)=x^d+r\in \Qi[x]$, and suppose that $r$ is non-zero and is not the reciprocal of an element of $\Zi$. Then $f$ is eventually stable over $\Qi$. 
\end{corollary} 

This paper is concerned with the case when $d=2$ and a non-zero $r\neq -1$ \emph{is} a reciprocal of an element of $\Zi$. In this case, "Eisenstein-type" methods break down. The following conjecture attempts to fully describe the situation. It is an extension of \cite[Conjecture 1.7]{evstb} which considered $r\in \Q$, but negatives have turned into squares, while the $c$-factorization $4m^2(m^2-1)$ for $f^2-$reducibility over $\Z$ expands to $c=\alpha^2(2i-\alpha^2)$:

%Conjecture 1.7
\begin{conjecture}
\label[conjecture]{conj1.7}
%(Demark-Hindes-Jones-Misplon-Stoll-Stoneman, \cite[Conjecture 1.7]{evstb}) 
Let $f_r=x^2+r$ with $r=1/c$ for $c\in \Zi\backslash\{0,-1\}$, and let $k_n$ denote the number of irreducible factors of $f_r^n(x)$. Then $f_r$ is eventually stable over $\Qi$ with constant $C(f_r,0)=4$:
\begin{enumerate}
\item If $c=\alpha^2$ with $1\pm i\alpha\in \Zi\backslash\Zi^2$ then $k_n=2$ for all $n\ge 1$.\label{itm:m2}
\item If $c\in \{\pm 8i, -16\}\subset \Zi^2$, then $k_1=k_2=2$, $k_n=3$ for all $n\ge 3$.\label{itm:8i}
\item If $c= (i\sigma^2-i)^2$ for $\pm \sigma \in \Zi\backslash\{3,5,56\}$ then $k_1=2$ and $k_n=3$ for all $n\ge 2$.\label{itm:neq3556}
\item If $c= (is^2 - i)^2$ for $\pm s\in\{3,5,56\}$ then $k_1=2$, $k_2 =3$ and $k_n=4$ for all $n\ge 3$. \label{itm:3556}
\item If $c\neq 48$ and $c=\alpha^2(2i-\alpha^2)$ for some $\alpha \in \Zi$, then $k_1 =1$ and $k_n=2$ for all $n\ge 2$.\label{itm:neq48}
\item If $c=48$, then $k_1=1, k_2=2$, and $k_n=3$ for all $n\ge 3$. \label{itm:48}
\item If $c\in \Zi\backslash\Zi^2$ and $c\neq \alpha^2(2i-\alpha^2)$ for any $\alpha \in \Zi$ (i.e. $c$ is not in an above case), then $k_n=1$ for all $n\ge 1$. \label{itm:irred}
\end{enumerate}
\end{conjecture}

%The proofs of cases \ref{itm:m2}-\ref{itm:48} (up to $N_{\C/\R}(c)\le10^9$) are the $\Zi$ analogue of \cite[Theorem 1.5]{evstb}). However, the proofs mimic their analogue in \cite{evstb} or use one of the strategies seen there, so they are omitted here for brevity. \\

All irreducibility statements from here on will be over the field $K$, and $r=1/c$ where $c\in \Zi$. Case \ref{itm:neq3556} and \ref{itm:3556} are where $c=\alpha^2$ with $1\pm i\alpha \in \Zi^2$.  Case \ref{itm:irred} is precisely the case where $f^2_r(x)$ is irreducible and thus case \ref{itm:irred} asserts that if $f^2_r(x)$ is irreducible, then $f^n_r(x)$ is irreducible for all $n\ge 1$. This is the main focus of this paper, and we state this as a $\Zi$ analogue of \cite[Conjecture 1.8]{evstb}:

%Conjecture 1.8
\begin{conjecture}
\label[conjecture]{conj1.8}
%(Demark-Hindes-Jones-Misplon-Stoll-Stoneman, \cite[Conjecture 1.8]{evstb} )\\
Let $f_r(x)=x^2+r$ with $r=1/c$ for $c\in \Zi \backslash \{0,-1\}$. If $f^2_r (x)$ is irreducible over $K$, then $f$ is stable over $K$.
\end{conjecture}

% The main result of this paper is a $\Zi$ analogue of \cite[Theorem 1.2]{evstb}:

% \begin{theorem}\label{mainthm}
% %(Demark-Hindes-Jones-Misplon-Stoll-Stoneman, \cite[Theorem 1.2]{evstb} )\\
% Let $K=\Qi$ and $f_r(x)=x^2+r$ with $r=1/c$ for $r=1/c$ for $c\in \Zi\backslash\{0,-1\}$. If $N(c)\le 10^9$, then $f_r$ is eventually stable over $\Qi$ and $C(f_r,0)\le 4$. In particular, \cref{conj1.7} above holds for all $c$ with $N(c)\le 10^9$.
% \end{theorem}

% %Theorem 3
% \begin{theorem}\label{thm1.5}
% %(Demark-Hindes-Jones-Misplon-Stoll-Stoneman, \cite[Theorem 1.5]{evstb})\\
% Let $f_r=x^2+r$ with $r=1/c$ for $c\in \Zi\backslash\{0,-1\}$, and let $k_n$ denote the number of irreducible factors of $f_r^n(x)$. Then,
% \begin{enumerate}
% \item[(1)] We have $k_1=k_2=2$ and $k_3 =3$ if and only if $c\in \{\pm 8i, -16\}$. In this case $k_n=3$ for all $n\ge 3$.
% \item[(2)] We have $k_3 \ge 4$ if and only if $c=-(s^2-1)^2$ for $\pm s\in \{3, 5, 56\}$. In this case $k_n=4$ for all $n\ge 3$.
% %\item[(3)] Suppose $N(c)\le 10^9$. We have $k_1=1$ and $k_3\ge 3$ if and only if $c=48$. In this case, and $k_2 =2$ and $k_n=3$ for all $n\ge 3$.
% \end{enumerate}
% \end{theorem}

 The determination of squares in the sequence $\{f_r^n(0)\}$ is of central importance in this study. To this end, we use tools such as the non-linear recurrence relation $$a_1=1, \quad a_n(c)=c^{2^{n-1}-1}+a_{n-1}^2(c) \quad \text{for } n\ge 2$$ describing the numerator of $f_r^n(0)=\frac{a_n}{c^{2^{n-1}}}$ for $n\ge 1$ (note $c^{2^{n-1}}$ is a square when $n\ge 2$).\\

We prove another irreducibility test involving the $a_n$ sequence. 
%Lemma1.10
\begin{lemma}
\label[lemma]{lem1.10}
%(Demark-Hindes-Jones-Misplon-Stoll-Stoneman, \cite[Lemma 1.10]{evstb} )
Suppose that $c\in\Zi\backslash\{0\}$, $r=1/c$ and $f_r^2$ is irreducible. Let $a_n=a_n(c)$ be defined as above and set
\begin{equation}
\label{bn}
b_n^\pm (c)=i(a_{n-1}\pm \sqrt{a_n})\in \bar{\Q}.
\end{equation}
If $b_n^\pm(c)$ is not a square in $\Zi$ (which holds in particular if $a_n(c)$ is non-square in $\Zi$) for all $n\ge 3$, then $f_r^n(x)$ is irreducible for all $n\ge 1$.
\end{lemma}
It is a $\Zi$ analogue of \cite[Lemma 1.10]{evstb}, but the lack of a linear ordering on $\Zi$ requires that we check that $b_n^-$ is non-square as well. This leads us to the following $\Zi$ extension of \cite[Conjecture 1.11]{evstb}:
%Conjecture 1.11
\begin{conjecture}
\label[conjecture]{conj1.11}
%(Demark-Hindes-Jones-Misplon-Stoll-Stoneman, \cite[Conjecture 1.11]{evstb} )\\
Let $b_n^\pm(c)$ be defined as in \ref{bn}. If $c\in \Zi\backslash\{0,-1\}$, then $b_n^\pm(c)$ is non-square in $\Zi$  for all $n\ge3$.
\end{conjecture}

The $\{b_n^\pm(c)\}$ square test of \cref{lem1.10} allows us to prove a $\Zi$ analogue of \cite[Theorem 1.3]{evstb}. This theorem is in some ways weaker than its counterpart in \cite{evstb} due to the lack of linear ordering on $\Zi$, however, it also expands on that theorem:
\begin{theorem}
\label[theorem]{THM1.3}
%(Demark-Hindes-Jones-Misplon-Stoll-Stoneman, \cite[Theorem 1.3]{evstb} )\\
Let $f_r(x)=x^2+r$ with $r=1/c$. Then $f_r^n$ is irreducible for all $n\ge 1$ if $c$ satisfies one of the following conditions: 
\begin{enumerate}
   \item $c, c+1\in \Zi\backslash\Zi^2$, $2\nmid c$ and $16\nmid c+1$\label{itm:no01.3};
    \item $c+1\in \Zi\backslash\Zi^2$, and $c\equiv \pm i\bmod a+bi$ with $N(a+bi)\equiv 5\mod 8$\label{itm:imodp1.3}
    \item $c\in \Zi\backslash\{0,-1, -1\pm 2i\}$ and $v_\pi(c)$ is odd for all primes $\pi\mid c$.\label{itm:odd1.3}
    %\item $c+1\in \Zi\backslash\Zi^2$ and $$\left|\frac{\prod_{2\nmid v_\pi(c) \text{ and } \pi\not\in 3 +4\Z} \pi^{v_\pi(c)}}{\prod_{2\mid v_\pi(c) \text{ or } \pi\in 3 + 4\Z} \pi^{v_\pi(c)}} \right|> 1.15/|c|^{1/30}.$$ \label{itm:ratio1.3}
\end{enumerate}
\end{theorem}

A notable consequence of \cref{itm:imodp1.3} is that for integers $c\equiv 2 \mod 4$, $f_r$ is stable over $\Zi$  
(and therefore $\Z$). Thus, by \cite[Corollary 4.6]{evstb},  $c\not\equiv 0 \mod 4$ implies $f_r$ is stable over $\Z$ (when $c$ is non-square in $\Z$). \\

%This offers more evidence for the conjecture that if $c\neq 4m^2(m^2-1)$, that $f^2$ irreducible implies $f^n$ is irreducible for all $n\ge 1$ (over $\mathbb{Z}$).\\

We prove another result allowing us to check irreducibility of $f^k$ for some finite $k$ to determine stability of $f$. The lack of a linear ordering on $\Zi$ leads to a new function $\Xi(c)$ for our bound, rather than $\epsilon(c)$ of \cite[Theorem 1.4]{evstb}. This also leads to a smaller denominator $\log\sqrt{1+\frac{1}{|c|}} < \log (1+\frac{1}{\sqrt{|c|}})$:

\begin{theorem}
\label[theorem]{THM1.4}
Let $f_r(x)=x^2+r$ with $r=1/c$ for $c\in \Zi$ with $|c|\ge 5$. If $f^k$ is irreducible for $$k=1+\left\lfloor\log_2 \left( 1+\frac{\log{4}+\Xi(|c|)/\sqrt{|c|}}{\log{\sqrt{1+1/|c|}}}\right)\right\rfloor,$$
then all $f^n$ are irreducible.  Here $\Xi(|c|)$ is a monotonically decreasing function on $\mathbb{R}_{\ge 5}$ with $$\sqrt{2}< \Xi(|c|) \le \Xi(5)< 6.05 \text{(see \cref{def4.5})}$$
\end{theorem}

Now, we define the orbit of $t\in \Qi$ under $f_r$ to be the set $O_{f_r}(t)=\{t, f_r(t), f_r^2(t),...\}$, and we say that a prime $\pi$ divides $O_{f_r}(t)$ if there is at least one nonzero $\alpha\in O_{f_r}(t)$ with $v_\pi(\alpha)>0$. We prove the following $\Zi$ analogue of \cite[Proposition 6.1]{evstb}, showing that the natural density of prime divisors of orbits is 0 when $b_n^\pm$ is non-square for all $n\ge 3$. The difficulty in a proof of this over $\Qi$ is constructing the correct tower of Galois extensions on which to apply Chebotarev's theorem:

%Another conjecture of Jones states that if $f\in \Z[x]$ is monic, degree 2, stable and its critical point $\gamma$ has an infinite orbit $O_{f_r}(\gamma)$, then the density of prime divisors of an integer orbit $O_{f_r}(m)$ is zero. \\

\begin{theorem}
\label[theorem]{thm6.1}
 %(Demark-Hindes-Jones-Misplon-Stoll-Stoneman, \cite[Proposition 6.1]{evstb})
Let $f_r(x)=x^2+r$ for $r=1/c$ for $c\in \Zi\backslash\{0,-1\}$, with $c,c+1$ both non-square in $\Zi$. If $\{|b_n^+|^2, |b_n^-|^2\}$ contains no squares in $\Z$ for all $n\ge 3$, then for any $t\in \Qi$ we have $$D(\{\pi\text{ prime in }\Zi: \pi \text{ divides } O_{f_r}(t)\})=0.$$
 \end{theorem}
 
\let\cleardoublepage\clearpage
\section{Background}
\quad We first list some work-saving facts:
%Theorem 1
\begin{proposition}
\label[proposition]{prop1}
If $p(x)\in \Qi[x]$, then any factorization of $p$ into irreducible polynomials over $\Qi$ has the same number of factors as $\bar{p}(x)$, $p(x)$ with conjugated coefficients. In particular, the polynomial $f^n_r(x)$ has the same number of factors irreducible over $\Qi$ as $f^n_{\bar{r}}(x)$.
\end{proposition}
\begin{proof}
Since conjugation is a field automorphism of $\Qi$, it extends to a ring automorphism of $\Qi[z]$. Ring automorphisms preserve irreducibility.
\end{proof}

This allows us to only consider $c$ in quadrants 1 and 3, reducing our work load by "half".\\

We also will need some way of inferring irreducibility of iterated polynomials: the following (derived from Capelli's Lemma) translates irreducibility of the list $\{g(f^n(x))\}_{n\ge1}$ to a property of a sequence in $K$.
\begin{lemma}
\label[lemma]{lemma2.2}
(Demark-Hindes-Jones-Misplon-Stoll-Stoneman, \cite[Lemma 2.2]{evstb})
Let $K$ be a field with $char(K)\neq 2$, $g\in K[x]$ a monic, irreducible polynomial with $d:=\deg(g)\ge 1$, and $f(x)$ monic and quadratic with critical point $\beta$. If the set $$\{(-1)^d g(f(\beta))\} \cup \{g(f^n(\beta))\}_{n\ge 2}$$ contains no squares in $K$, then $g(f^n(x))$ is irreducible over $K$ for all $n\ge 1$.\\
\end{lemma}
% \begin{proof}
% Let $f(x)=x^2+bx+c$, so that $\gamma=-b/2$. We proceed by induction on $n$, with the $n=0$ case covered by the irreducibility of $g(x)$. Assume then that $g(f^{n-1}(x))$ is irreducible over $K$ for some $n\ge 1$, and let $d_1$ be the degree of $g(f^{n-1}(x))$. By Capelli's Lemma (\ref{capelli}), $g(f^n(x))$ is irreducible over $K$ if and only if for any root $\beta$ of $g(f^{n-1}(x))$, we have $f(x)-\beta$ is irreducible over $K(\beta)$, or equivalently (because $K$ has characteristic different from 2), $\textrm{Disc}(f(x)-\beta)=b^2-4c+4\beta$ is not a square in $K(\beta)$.\\
% This must hold if $N_{K(\beta)/K}(b^2-4c+4\beta)$ is not a square in $K$. The Galois conjugates of $b^2-4c+4\beta$ are precisely $b^2-4c+4\alpha$ as $\alpha$ varies over all roots of $g(f^{n-1}(x))$. Thus \begin{equation*}
%     \begin{split}
%         N_{K(\beta)/K}(b^2-4c+4\beta) &= (-4)^{d_1} \prod_{\alpha \text{ root of } g\circ f^{n-1}} \left[ \left( -\frac{b^2}{4}+c\right)-\alpha\right]\\
%         &= (-4)^{d_1} \cdot g(f^{n-1}(-b^2/4+c))\\
%         &= (-4)^{d_1}\cdot g(f^{n-1}(\gamma))),
%     \end{split}
% \end{equation*}
% where the second equality holds because $g(f^{n-1}(x))$ is monic. Now $d_1$ is odd if and only if $n=1$ and $d$ is odd, which proves the Lemma.
% \end{proof}
\begin{remark}
    The proof found in \cite[Lemma 2.2]{evstb} shows that if $g(f^n(x))$ is irreducible (over $K$) for some $n\ge 1$ and $g(f^{n+1}(\beta))$ is non-square in $K$, then $g(f^{n+1}(x))$ is irreducible (over $K$).

    When $g(x)=f_r(x)=x^2+r$, the hypothesis asks if the set $\{f_r^n(0)=\frac{a_n(c)}{c^{2^{n-1}}}\}_{n\ge 1}$ contains no squares in $K$ $\iff$ the set $\{a_n=a_{n-1}^2(c)+c^{2^{n-1}-1}\}_{n\ge 1}$ contains no squares in $\Zi$.
\end{remark}

%  \begin{remark}
%     Since $$-m^2(m^2+2i)= \overline{-\bar{m}^2(\bar{m}^2-2i)},$$ by \cref{prop1} we may restrict our study to $-m^2(m^2-2i)$, since $k_n$ for $f^n_{\bar{r}}$ (the number of factors over $\Qi$ when $\bar{r}=1/\bar{c}$) is the same as $k_n$ for $f^n_r$.\\
% \end{remark}
%\subsection{When $f_r$ or $f^2_r$ is reducible}

\quad This idea that we may infer irreducibility of $f^n$ by considering whether $f^n(0)$ is a square leads to necessary and sufficient conditions for the irreducibility of $f_r$ and $f^2_r$ when $f_r$ is irreducible. The following specialization %of the above lemma is a $\Zi$ analogue of \cite[Proposition 2.1]{evstb}.
is one of the first instances where the lack of ordering changes our results:%If $c=\alpha^2(2i-\alpha^2)$, then 

%$$f^2_r(x)= \left(x^2-\frac{1+i}{\alpha}x- \frac{i\alpha^2+1}{\alpha^2(2i-\alpha^2)}\right)\left(x^2+\frac{1+i}{\alpha}x-\frac{i\alpha^2+1}{\alpha^2(2i-\alpha^2)}\right).$$ 

\begin{proposition}
\label[proposition]{prop2.1}
Let $f_r(x)=x^2+r$ with $r=1/c$ for $c\in \mathbb{Z}[i]\backslash \{0\}$. Then $f_r(x)$ is reducible if and only if $c=\alpha^2$ for some $\alpha \in \Zi$. If $f_r(x)$ is irreducible, then $f^2_r(x)$ is reducible if and only if $c=\alpha^2(2i-\alpha^2)$ for some $\alpha\in \mathbb{Z}[i]$.\\
\end{proposition}

 \begin{proof} %The proof is analogous to \cite[Proposition 2.1]{evstb}.

%To prove the first statement, note that $f_r(x)=x^2+r = (x+i\sqrt{r})(x-i\sqrt{r})$ over $\Qi$ if and only if $\pm i\sqrt{r}\in \Qi$, if and only if $r$ is square in $\Qi$. Then $c=1/r$ is a square in $\mathbb{Z}[i]$.

%%%

The first statement is clear. Now suppose $f_r$ is irreducible over $\Qi$. If $c=\alpha^2(2i-\alpha^2)$, then 

$$f^2_r(x)= \left(x^2-\frac{1+i}{\alpha}x- \frac{i\alpha^2+1}{\alpha^2(2i-\alpha^2)}\right)\left(x^2+\frac{1+i}{\alpha}x-\frac{i\alpha^2+1}{\alpha^2(2i-\alpha^2)}\right).$$

These factors are reducible if and only if their discriminant $\frac{-2i}{\alpha^2-2i}$ is square in $\Qi$ (i.e. $\alpha^2- 2i$ is a square in $\Zi$): suppose $\alpha^2-2i=\eta^2$. Then $2i= (\alpha-\eta)(\alpha+\eta)$. If $\{\alpha+\eta, \alpha-\eta\} =\{-1+i,1-i\}$, this implies $2\alpha=0$. Since $c\neq 0$, we have a contradiction. We obtain similar contradictions when $\{\alpha+\eta, \alpha-\eta\}\in \pm\{ \{1+i,1+i\}, \{1, 2i\}, \{i, 2\}\}$. Thus the factors of $f^2_r(x)$ are irreducible over $\Qi$ when $c=\alpha^2(2i-\alpha^2)$.\\

The converse is analogous to the proof of \cite[Proposition 2.1]{evstb}, but $2(k^2-1)(-1\pm k)$ is potentially a square in $\Zi$ instead of only $ 2(k^2-1)(-1+ k)$ over $\Q$ (where $c=k^2-1$). We also need to factor $2=i(1-i)^2$.

\end{proof}

Non-squares in $K=\Qi$ are thus of great interest, so squares in $\Zi$ are too. A tactic is checking for quadratic non-residues: if the Gaussian integer $\alpha$ is non-square mod $\pi$, then $\alpha$ is non-square in $\Zi$. In the rest of this paper, $\glegendre{\cdot}{\alpha}$ and $\glegendre[2]{\cdot}{\pi}$ refer to the Gaussian integer Jacobi and Legendre symbols (respectively) as in \cite{recip}. Thus if $\alpha=\pi_1^{a_1}\cdot\pi_2^{a_2}\cdots \pi_n^{a_n}$ is a prime factorization of $\alpha\in \Zi$, then $$\glegendre{\sigma}{\alpha}=\glegendre[2]{\sigma}{\pi_1}^{a_1}\cdot\glegendre[2]{\sigma}{\pi_2}^{a_2} \cdots \glegendre[2]{\sigma}{\pi_n}^{a_n}=
    -1 \quad \text{ if } \sigma \text { is a non-residue }\bmod \alpha.$$

The notion of a "rigid divisibility sequence" will also be advantageous in this search for non-squares.
\subsection{Rigid divisibility}
\quad One can show inductively that $f_{1/c}^n(0)=a_n/c^{2^{n-1}}$, where \begin{equation*}
   a_1(c)=1, \quad a_n(c)=a^2_{n-1}(c)+c^{2^{n-1}-1} \quad \text{ for } n\ge 2. 
\end{equation*} 
Note that the numerator $a_n$ is relatively prime to the denominator $c^{2^{n-1}}$, a square when $n>1$. This numerator inherits prime divisors in a predictable way, leading to the notion of "rigid divisibility":
%a property of some sequences $\{z_n\}_{n\ge 1}$ yielding that the prime divisors of $z_j$ are related to those of $z_{jk}$ in a precise way:

%Definition 2.5
\begin{definition}
\label[definition]{def2.5}
Let $A=\{z_n\}_{n\ge 1}$ be a sequence in a field $K$. We say $A$ is a \textbf{rigid divisibility sequence over $K$} if for each non-archimedean absolute value $|\cdot|$ on $K$, the following  hold:\\

(1) if $|z_n|<1$, then $|z_n|=|z_{kn}|$ for any $k\ge 1$.\\

(2) if $|z_n|<1$ and $|z_j|<1$, then $|z_{\text{gcd}(n,j)}| <1$.\\
\end{definition}

The following theorem establishes rigid divisibility of $\{f^n(0)\}_{n\ge 1}$, and implies that the $\{a_n\}_{n\ge1}$ sequence also has this property:
\begin{theorem}
\label[theorem]{anrigdiv}(Hamblen, Jones and Madhu, \cite[Lemma 12]{orbits})
    Let $K$ be a field and $f(x)=x^d+r\in K[x]$ for some $d\ge 2$. Then $\{f^n(0)\}_{n\ge 1}$ is a rigid divisibility sequence over $K$.
\end{theorem}

The notion of rigid divisibility leads to tables of congruences for $c$ ensuring $f_{1/c}^n$ is irreducible for all $n$, the $\Zi$ analogue of \cite[Proposition 3.5]{evstb}:
%Proposition 3.5
\begin{proposition}
\label[proposition]{prop3.5}
Suppose that $c\in \Zi\backslash\{0\}$. If $c$ satisfies any of the congruences in \Cref{tab:2a}, or $c+1\in \Zi\backslash\Zi^2$ and $c$ satisfies any of the congruences in \Cref{tab:2b}, then $a_n$ is not a square in $\Zi$ for all $n\ge2$.
\end{proposition}

\begin{center}
\begin{tabularx}{\textwidth}{Xl} 
\caption{Congruences that ensure $a_n$ is non-square for $n\ge2$}\label{tab:2a}
 \\ \hline
 $c\equiv i$, $1+i$ & $(\text{mod } 2)$ \\

$c\equiv -i$ & $(\text{mod } 2+i)$ \\
$c\equiv -1-i$ & $(\text{mod } 3+i)$ \\
$c\equiv i$, $\pm1 +i$ & $(\text{mod } 3+2i)$ \\
$c\equiv 1+2i$ & $(\text{mod } 4)$ \\

% $c\equiv 2- i$ & $(\text{mod } 4+i)$ \\
% $c\equiv i$, 2, -3, $-(1+i)$, $1-2i$ & $(\text{mod } 4+3i)$ \\
% $c\equiv 1$ & $(\text{mod } 4+4i)$\\
% $c\equiv -2+2i$, $-1-i$, $1-2i$  & $(\text{mod } 5)$ \\
% $c\equiv -2+i$, $\pm1-i$, $-i$, $1+2i$, $2+2i$ & $(\text{mod } 5+i)$  \\
% $c\equiv -i$, 2, $2+i$ & $(\text{mod } 5+2i)$ \\
% $c\equiv -2+2i$, $2+i$ & $(\text{mod } 5+3i)$ \\
% $c\equiv 1+3i$ & $(\text{mod } 5+4i)$ \\
% $c\equiv -2+2i$, $i$, $-3i$ & $(\text{mod } 6+i)$ \\
% $c\equiv -1 \pm 3 i$, $-i$, $2 - 3 i$ & $(\text{mod } 6+5i)$ \\
% $c\equiv 3\pm i$ & $(\text{mod } 7)$ \\

%$c\equiv -i, 1 + 2 i, 1 + 4 i, 2 + 3 i, 3 + i$ & $(\text{mod } 7+2i)$ \\
\\ \hline
\end{tabularx}
\begin{tabularx}{\textwidth}{Xl} 
\caption{Congruences that ensure $a_n$ is non-square for $n\ge2$, provided that $c+1$ is non-square.}\label{tab:2b}
 \\ \hline
$c\equiv i$ & $(\text{mod } 2+i)$ \\
$c\equiv -1\pm i$ & $(\text{mod } 3)$ \\
$c\equiv \pm i$, $\pm1+i$, $1-i$ & $(\text{mod } 3+i)$ \\
$c\equiv -i$, $-2$, $-2i$ & $(\text{mod } 3+2i)$ \\
$c\equiv -1+2i$, $\pm1-2i$ & $(\text{mod } 4)$ \\

$c\equiv -2$, $1+i$, $1+2i$, $2-i$ & $(\text{mod } 4+i)$ \\
$c\equiv \pm i$, $\pm2$, $\pm 3$, $\pm (1+i)$, $\pm(1-2i)$, & $(\text{mod } 4+3i)$ \\
$c\equiv 1$, 3 & $(\text{mod } 4+4i)$ \\
$c\equiv \pm(2-2i)$, $\pm(1+i)$, $\pm(1-2i)$ & $(\text{mod } 5)$ \\
$c\equiv -2, -2 \pm i, \pm1 - i, \pm i, 2 i, 1 \pm 2 i, \pm(2 + 2 i)$  & $(\text{mod } 5+i)$\\

\\ \hline
\end{tabularx}
\end{center}

\begin{proof}
Let $f(x)=x^2+1/c$. The sequence $\{f^n(0)=\frac{a_n(c)}{c^{2^{n-1}}}\}_{n>1}$ contains no squares in $\Qi$ if and only if the sequence $\{a_n(c)\}_{n>1}$ contains no squares in $\Zi$. When $a_n(c)$ is not a quadratic residue mod $\pi$, it must be non-square in $\Zi$. \Cref{tab:2a} consists of such pairs $z,\pi$ with $c\equiv z \mod \pi$ implying $a_n(c)$ is a non-residue mod $\pi$ for all $n\ge 2$: for example, when $c\equiv i \mod 2$, the relationship $a_k(c)=c^{2^{k-1}-1}+a_{k-1}(c)^2$ yields $$\{1, a_2=c+1, a_3, ... \}\equiv\{1, 1+i, i, 1+i, ...\} \bmod 2.$$   Since $\glegendre[2]{i}{2}=\glegendre[2]{1+i}{2}=-1$ (i.e. $i$ and $1+i$ are not quadratic residues mod 2), $a_k$ is non-square in $\Zi$ for all $k\ge 1$.\\

%and take $z$ to be a fixed Gaussian integer with $z\nmid c$. Let $k:= \Zi/z\Zi$, with $c_0\in k$ satisfying $(1/c)\equiv c_0 \bmod z$ and put $\bar{f}=x^2+c_0\in k[x]$. Now $a_n=c^{2^{n-1}}f^n(0),$ and it follows that if $\bar{f}^n(0)$ is not a square in $k$ for $n\ge 2$, then $a_n$ is not a square in $\Zi$. The sequence $(\bar{f}^n(0) \bmod z)_{n\ge 1}$ eventually lands in a repeating cycle. 

We now show that $a_p$ non-square implies $a_{kp}$ is non-square for $c\in\{0,1\}\bmod 2$ (since $c\in\{i,1+i\}\bmod 2$ are already handled). By rigid divisibility, $v_{\pi}(a_p)=v_{\pi}(a_{kp})$, so we can conclude that $a_{kp}$ is non-square in $\Zi$ if we know that $v_{\pi}(a_p)$ is odd.\\

When $c\equiv 0 \bmod 2$, and $p$ is an integer prime, $a_p$ either has a prime divisor of odd multiplicity or is of the form $i\alpha^2$ when it is non-square. To rule out the latter, note that $a_1(c)=1$ and for all $n\ge 2$, $a_n=a_{n-1}^2+c^{2^{n-1}-1}\equiv 1+0\equiv 1 \bmod 2.$ However, $i\alpha^2 \in \{0,i\}\bmod 2$ when $\alpha \in \Zi$. Since $a_n\equiv 1\bmod 2$ for all $n$, $a_p$ and thus $a_{kp}$ has this prime divisor to the same multiplicity, hence is non-square.\\

Now suppose $c\equiv 1 \bmod 2$. To show that $a_p\neq i\alpha^2$, we start with $p=2$. Then $a_2=c+1\equiv 0\bmod 2\implies a_{2n}\equiv 0 \bmod 2$ by rigid divisibility. We show $a_{2n}$ is non-square when $a_2$ is non-square: if $a_2=c+1$ has some prime divisor $\pi$ of odd multiplicity, rigid divisibility then implies that $a_{2n}$ is non-square for all $k\ge 1$.
Now suppose $c+1=i\alpha^2$. Since $a_{2n}(c)=(c+1)P_{2n}(c)$ over $\Z$, the only way $a_{2n}(c)$ is a square in $\Zi$ is if $P_{2n}(c)=i\alpha_1^2$ for some $\alpha_1 \in \Zi$. However, by rigid divisibility $v_{1+i}(a_{2n})=v_{1+i}(c+1)$, so $1+i\nmid P_{2n}(c)$. Since $P_{2n}(c)\in \Z[c]$, $1+i \nmid P_{2n}(c)$, and $c\equiv 1\bmod 2$, we must have $P_{2n}(c)\equiv 1\bmod 2$. Thus $P_{2n}(c)\neq i\alpha_1^2$, so $a_{2n}(c)$ is non-square.\\

Also, $a_{2n+1}\equiv c^{2^{2n}-1}+a_{2n}^2\equiv 1 \bmod 2$, so if $a_{2n+1}$ is non-square in $\Zi$, $a_{2n+1}\neq i\alpha^2$. We have thus shown $a_p(c)$ non-square implies $a_{kp}(c)$ is non-square when $c\in \{0,1\}\mod 2$.\\

The rest proceeds analogously to the proof of \cite[Proposition 3.5]{evstb}.

\end{proof}
\begin{remark}
    Since $c\equiv i \mod 2$ and $c\equiv 1+i \mod 2$ are in \cref{tab:2a}, $f_r$ is stable when $f_r^2$ is irreducible for such $c$. We may assume $c\in  \{0,1\} \mod 2$ for the rest of the paper.
\end{remark}

The proof of \cite[Proposition 3.3]{evstb} is also easily adapted to the Gaussian integer setting, and rigid divisibility extends this result to the claim that $a_{3n}$ is non-square in $\Zi$ for all $n\ge 1$:

%Proposition 3.3
\begin{proposition} 
\label{prop3.3}
If \
$c\in \Zi\backslash\{0, -1, \pm 2i, -89\}$, then $a_3$ is non-square in $\Zi$.
\begin{proof}
    We have $a_3(c)=c^3+c^2+2c+1$, and so if $a_3(c)=y^2_0$ for $y_0\in\Zi$, then necessarily $(c,y_0)$ is an integer point on the elliptic curve $y^2=x^3+x^2+2x+1$. This curve has conductor norm 2116 with label \href{https://www.lmfdb.org/EllipticCurve/2.0.4.1-2116.1-b2}{2.0.4.1-2116.1-b2} in the LMFDB \cite{lmfdb}. It has Mordell-Weil group generator $(-1:i:1)$ and torsion isomorphic to $\Z/3\Z$ with generator $(0:-1:1)$. We determine all $\Zi$-integral points using MAGMA \cite{magma} with code from Thongjunthug's thesis \cite{ellip}, :
    
    $$\{(x,\pm y)\} = \{(-1, i), (0,1), (-2i, 1 + 2i), (2 i, 1 - 2i), (-89, 835i)\}$$

    Note $c\in \{0,-1,\pm 2i, -89\}$ is excluded by hypothesis.
\end{proof}
\end{proposition}
\section{The proof of Theorem \ref{THM1.3}}
\quad The Gaussian integer analogue of Lemma 3.2 of \cite{evstb} is another important instance of the lack of a total ordering on $\Zi$ coming into play, with an identical proof (until the last line):
%Lemma 3.2
\begin{lemma} \label[lemma]{lem3.2} 
    Suppose that $r=1/c$ and $f^2_r(x)$ is irreducible. Let $a_n=a_n(c)$ and $$b_n^\pm(c)=i(a_{n-1}\pm \sqrt{a_n}).$$ If for every $n\ge 3$, $\{b_n^+(c), b_n^-(c)\}$ contains no squares in $\Zi$ (which holds if $a_n$ is non-square in $\Zi$), then $f_r^n(x)$ is irreducible for all $n\ge 1$.
\end{lemma}

\begin{remark}
    The proof of the above \cite[Lemma 3.2]{evstb} shows that if $f^n_{1/c}(x)$ is irreducible and the set $\{b_{n+1}^+(c), b_{n+1}^- (c)\}$ contains no squares in $\Zi$, then $f^{n+1}_{1/c}(x)$ is irreducible for $n\ge2$.\\
\end{remark}

The following is a $\Zi$ analogue of \cite[Theorem 3.6]{evstb} giving simple criteria for $b_n^\pm(c)$ to be non-square for all $n\ge 3$:
%Theorem 3.6
\begin{theorem} \label[theorem]{thm3.6}
Let $f_r(x)=x^2+r$ with $r=1/c$ for $c\in \Zi\backslash\{0\}$. Let $a_n$ and $b_n^\pm$ as in \cref{bn}. Assume that $c$ satisfies one of the following conditions:

    \begin{enumerate}
        \item $c, c+1\in \Zi\backslash\Zi^2$, $2\nmid c$ and $16\nmid c+1$; \label{itm:1mod2}
        \item $c+1\in \Zi\backslash\Zi^2$, and $c\equiv \pm i\bmod a+bi$ with $N(a+bi)\equiv 5\mod 8$; \label{itm:imodp1}
        %\item $c\equiv -i\bmod p$ for a prime $p$ with $N(p)\equiv 5\bmod 8$ with $\Real(p)+\Imag(p)\equiv 3, 5 \bmod 8$; \label{itm:imodp2}
        %$c\equiv -i \bmod p$ for a prime $p\equiv 1+2i, 2+i\bmod 4$. 
        %$c\equiv i\bmod p$ for a prime $p\notequiv -1+2i, -2+i\bmod 4$, or $c\equiv -i \bmod p$ for a prime $p\equiv 1+2i, 2+i\bmod 4$. 
        \item $c\in \Zi\backslash\{0,-1, -1\pm 2i\}$ and $v_\pi(c)$ is odd for all primes $\pi$ dividing $c$. \label{itm:odd}
        %\item $v_p(c)$ is odd for some prime $p\neq 1+i$ dividing $c$. \label{itm:podd}
    \end{enumerate}
In \cref{itm:imodp1}, $a_n$ is non-square for $n\ge 2$, while in \cref{itm:1mod2,,itm:odd}, $\{b_n^+(c),b_n^-(c)\}$ contains no squares in $\Zi$ for all $n\ge 3$. In all cases, $f_r^n$ is irreducible for all $n\ge 1$.
\end{theorem}
We break this theorem into three lemmas.

\subsubsection{Case (1)}
\begin{lemma}\label[lemma]{1mod2}
Suppose $c, c+1\in \Zi\backslash\Zi^2$. If $2\nmid c$ and $16\nmid c+1$, then 
$f_r^n$ is irreducible for all $n\ge 1$.
\end{lemma}

\begin{proof}
By \Cref{tab:2a}, we may assume $c\not\in\{i,1+i\} \bmod 2$. Note $f$ is irreducible since $c\not\in \Zi^2$ by \cref{prop2.1} and $a_2=c+1$ is non-square in $\Zi$. By the remark after \cref{lemma2.2}, we have that $f_r^2$ is irreducible.
%and $c\notin \Zi^2$ implies $c+1=(i\alpha^2+1)^2$.

Since $a_2=c+1$ is non-square by hypothesis and $c\in \{0,1\} \bmod 2$, note $a_{2n}$ (and hence $b_{2n}^\pm$) are also non-square by the proof of \cref{prop3.5}. We need only show $b_{2n+1}^\pm$ is non-square in $\Zi$.

%Note that $f_r$ is irreducible since $c$ is non-square mod $2$, hence $r=1/c$ is non-square in $\Qi$. For $c\equiv i \bmod 2$ we have the sequence $$\{a_n \bmod 2\}_{n\ge1}= \{1, 1+ i,  i,  1 + i, i, ...\} \bmod 2$$ and for $c\equiv 1+i \bmod 2$, we have the sequence $$\{a_n \bmod 2\}_{n\ge1}=\{1, i,  i, ...\}$$ indicating that $a_n$ is non-square mod 2 for all $n\ge 2$, thus $a_n$ is non-square in $\Zi$ for all $n\ge 2$. \\

Since $2\nmid c$ and $c\not\in \{i,1+i\}\bmod 2$, $c\equiv 1\bmod 2$. By the recurrence relation $a_k=a_{k-1}^2+c^{2^{k-1}-1}$, we must have $a_k\in \{0,1\} \bmod 2$ for all $k\ge 1$. If $a_k$ is non-square then $b_k^\pm$ is as well, so suppose $a_k$ is square. If $\sqrt{a_k}\in\{0,1\}\bmod 2$, then $$\sqrt{a_k}\equiv a_k \equiv a_{k-1}^2 + c^{2^{k-1}-1}\equiv a_{k-1} +1 \bmod 2.$$
Now,
 $$b_k^\pm \equiv i(a_{k-1} \pm \sqrt{a_k}) \equiv i(a_{k-1} + (a_{k-1} +1)) \equiv i \bmod 2,$$ so $b_k^\pm$ is non-square modulo 2, hence is non-square in $\Zi$. We now suppose $\sqrt{a_k} \in \{i, 1+i\} \bmod 2$. \\
 
%Now suppose $\sqrt{a_n}\in \{i, 1+i\} \bmod 2$.\\

%If $c\equiv 0 \bmod 2$, then $a_n= c^N + a_{n-1}^2(c) \equiv 1\bmod 2^N$ in this case. Then $\sqrt{a_n} \in \{1,i\} \bmod 2$, so suppose $\sqrt{a_n}\equiv i \bmod 2$, or else $b_n^\pm$ is non-square as shown above. Now $a_n \equiv 3\bmod 4$. This is impossible when $n\ge 3$ since $N=2^{n-1}-1 \ge 3$ in this case, so $a_n\equiv 1\bmod 8$. Thus $n=2$ is the only opportunity for $a_n\equiv 3\bmod 4$. Thus for $n\ge 3$, $\sqrt{a_n}\equiv 1 \bmod 2$, so $b_n^\pm \equiv i \bmod 2$ as before. Hence $b_n^\pm$ is non-square for all $n\ge 3$. \\

Note  $c\equiv 1\bmod 2$ implies $c\in \{\pm 1, \pm1 + 2i\} \bmod 4$ and $c+1\in \{0, 2, 2i, 2+ 2i\}\bmod 4$. Then $c^2\equiv 1 \bmod 4$, so $a_k=c^{2^{k-1}-1}+a_{k-1}^2 \equiv c+ a_{k-1}^2 \bmod 4$. By rigid divisibility of $\{a_k\}_{k\ge 1}$, we have $2\mid c+1 \implies 2\mid a_{2k}\implies 4\mid a_{2k}^2$, so combining $a_k\equiv c+a^2_{k-1}\bmod 4$ with this fact yields \begin{equation}\label{eq:anmod4}
    \{a_k\}_{k\ge1} \equiv \{1, c+1, c, c+1, c,...\}\bmod 4
\end{equation} \\

First suppose $\sqrt{a_k} \equiv 1+i \bmod 2$, so $a_k\equiv 2i \bmod 4$.
%Since $1+i\mid a_k$, we must have $1+i\nmid c$ since $\gcd(c,a_k)=1$ for all $k$. 
Since $c\equiv 1\bmod 2$, $c+1\equiv a_{2n} \equiv 2i \bmod 4$ by \cref{eq:anmod4}. Thus $c\equiv -1+2i\equiv a_{2n+1}\bmod 4$, so $c$ and $a_{2n+1}$ are non-square as well. Thus $a_k$ is non-square for all $k\ge 2$ when $\sqrt{a_k} \equiv 1+i \bmod 2$.\\

Now suppose $\sqrt{a_k}\equiv i\bmod 2$. Since $\sqrt{a_k}=2z+i$, we have $a_k\equiv -1 \bmod 4$. Since $c+1\equiv 0\bmod 2$, \cref{eq:anmod4} implies $c\equiv a_k\equiv -1\bmod 4$.  Then $c\in \{-1, -1+4i, 3, 3+4i\} \bmod 8$, so $c+1\in \{0, 4i, 4, 4+4i\} \bmod 8$. Again \begin{equation}\label{eq:anmod8}
    \{a_k\}_{k\ge1} \equiv\{1, c+1, c, c+1, c,...\}\bmod 8
\end{equation} 
 by analogous reasoning to \cref{eq:anmod4}.\\

 If $c\in \{-1+4i, 3\}\bmod 8$ we're done, these are non-square and would imply $a_{2k+1}$ is non-square in $\Zi$ by \cref{eq:anmod8}, thus $b_{2k+1}^\pm(c)$ is not a square in $\Zi$.
 
 If $c\equiv 3+4i \equiv(2+i)^2 \bmod 8$, then $a_{2k+1}\equiv 3+4i\bmod 8$ by \cref{eq:anmod8} and $$\sqrt{a_{2k+1}}\in \{-2-i,-2+3i, 2+i,2-3i\}\bmod 8$$ when $a_{2k+1}$ is a square in $\Zi$. Also, $c+1\equiv a_{2k}\equiv 4+4i\bmod 8$ and thus $$b_{2k+1}^\pm(c) \equiv i(4+4i \pm \sqrt{3+4i})\in \{-3+2i,-7+2i,-5+6i, -1+6i \} \bmod 8.$$
Since this set consists of non-squares mod 8, $b_{2k+1}^\pm$ is not a square in $\Zi$ when $c\not\equiv -1\bmod 8$.\\

If $c\equiv -1\bmod 8$, then $c\equiv a_{2k+1}\bmod 8$ and $a_{2k}\equiv c+1 \equiv 0\bmod 8$ by \cref{eq:anmod8}. Therefore $\sqrt{a_{2k+1}}\in \{\pm i, \pm 3i\}\bmod 8$. If $\sqrt{a_{2k+1}}\equiv \pm3i\bmod 8$, then we have $$b_{2k+1}^\pm(c) \equiv i(a_{2k}\pm \sqrt{a_{2k+1}})\equiv \pm3\bmod 8$$ and is thus non-square. Thus we may assume $\sqrt{a_{2k+1}}\equiv \pm i\bmod 8 \implies a_{2k+1}\equiv -1\bmod 16$. Also, $$\{a_k\}_{k\ge 1} \equiv \{1, c+1, c, c+1, ...\}\bmod 16$$ again by the same reasoning as \cref{eq:anmod4}.
%since $c$ has order 2 mod 16, and since $8\mid c+1\implies 8\mid a_{2k}$ by rigid divisibility of $\{a_k\}$. 
Thus $c\equiv a_{2k+1}\equiv -1 \bmod 16$. \\

Now, when $c,c+1\in \Zi\backslash\Zi^2$, and $2\nmid c$, we find that $c\not\equiv -1 \bmod 16$ must land into an above case. Thus $\{b_n^+ (c),b_n^-(c)\}$ contains no squares in $\Zi$ for each $n\ge 3$. Since $f^2_r$ is irreducible, $f_r^n$ is irreducible for all $n\ge 1$ by \cref{lemma2.2}.\\

\end{proof}

\begin{remark}
If $c\equiv -1 \bmod 16$, then $\sqrt{a_{2k+1}}\in \pm\{i, 7i\} \bmod 16$ with $$b_{2k+1}^+(c)\equiv i(a_{2k}+\sqrt{a_{2k+1}})\equiv \begin{cases}
    & i(0\pm 7i)\equiv \pm 7\\
    & i(0 \pm i)\equiv \pm 1\\
\end{cases}\bmod 16$$ (the same is true of $b_{2k+1}^-$) but $\pm1$ and $\pm 7$ are all squares mod 16: note $-1=i^2$, while $7\equiv (3i)^2 \bmod 16$. Thus the method of proof of \cref{1mod2} cannot conclude that $b_{2k+1}^\pm (c)$ is non-square for such $c$.
\end{remark}
\subsubsection{Case (2)}
\begin{lemma}
If $c+1\in \Zi\backslash\Zi^2$, and $c\equiv \pm i\bmod a+bi$ with $N(a+bi)\equiv 5\mod 8$, then 
$\{a_n\}_{n\ge 2}$ contains no squares in $\Zi$. Thus $f_r^n$ is irreducible for all $n\ge 1$. 
\end{lemma}
\begin{proof}

Let $f_r(x)=x^2+r$ with $r=1/c$ for $c\in \Zi\backslash\Zi^2$. If $c\equiv i\bmod a+bi$ then \begin{equation}\label{eq:ana+bi}
    \{a_n(c)\bmod a+bi \}_{n\ge 1}\equiv  \{1, 1+i, i, -1-i, i, -1-i, ...\}.
\end{equation}
 
Note that $\glegendre{\pm i}{a+bi}=(-1)^{(a^2+b^2-1)/4}$ when $a$ is odd, $b$ is even by \cite[Theorem 17]{recip}. I claim the same is true when $a$ and $b$ are reversed: note that 
\begin{alignat*}{2}
& i=\gamma^2+z(a+bi) \iff -i &&=\bar{\gamma}^2 +\bar{z}(a-bi)\\
& &&=\bar{\gamma}^2 +\bar{z}(-i\cdot i)(a-bi)\\
& &&= \bar{\gamma}^2 -i\bar{z}(b+ai),
\end{alignat*}
    
so 
\begin{equation}\label{eq:pmi}
\begin{split}
        \glegendre{\pm i}{a+bi}=\glegendre{\pm i}{b+ai}=(-1)^{(a^2+b^2-1)/4}
    \end{split}
\end{equation} when $a$ and $b$ have opposite parity. \\

Thus $N(a+bi)=a^2+b^2\equiv 5\mod 8$ implies $\glegendre{\pm i}{a+bi}=-1$ by \cref{eq:pmi}. Thus $a_{2k+1} \equiv i \bmod a+bi$ is non-square for all $k\ge 1$ by \cref{eq:ana+bi}.\\

Since $c+1=a_2$ is non-square by hypothesis, the proof of \cref{prop3.5} shows that $a_{2n}$ is non-square for all $n\ge 1$. Thus $a_k$ is non-square for all $k\ge 1$ when $c\equiv i\bmod a+bi$. \\

In the case $c\equiv -i\bmod a+bi$ we have $$\{a_n(c)\}_{n\ge 1}\equiv \{1, 1-i, -i, -1+i, -i, -1+i, ...\}\bmod a+bi,$$
so the claim is still true following the above proof.\\ 

Since $c\equiv \pm i \bmod a+bi$ by hypothesis, $c$ is non-square mod $a+bi$, so $c$ is non-square in $\Zi$. Thus $f_r$ is irreducible. Since $\{a_n\}_{n\ge 2}$ contains no squares in all cases, $f^n_r$ is irreducible for all $n\ge 1$ by \cref{lemma2.2}. \\

\end{proof}

\begin{remark}
    This theorem can be used to show that if $c\in \Z$ with $c\equiv 2\bmod 4$, then $f_{1/c}^n$ is irreducible for all $n$. Note $c\equiv -i\bmod c+i$ and $N(c+i)=c^2+1 \equiv 5 \mod 8$. The hypothesis that $c+1$ is non-square is satified since $c+1\equiv 3\mod 4$, but squares in $\Z$ are $0$ or $1\bmod 4$.
\end{remark}
%3case
\subsubsection{Case (3)}
\begin{lemma}\label[lemma]{case3}
If $c\in \Zi\backslash\{0,-1, -1\pm 2i\}$ and $v_\pi(c)$ is odd for all primes $\pi$ dividing $c$, then 
 $f_r^n$ is irreducible for all $n\ge 1$. %then 
% $\{b_n^+(c)\}_{n\ge 3}\cup \{b_n^-(c)\}_{n\ge 3}$ contains no squares in $\Zi$, so .
\end{lemma}
\begin{proof}
By \Cref{tab:2a}, we may assume $c\in\{ 0,1\} \bmod 2$. Also, note that $c\neq \alpha^2(2i-\alpha^2)$ unless $\alpha$ is a unit in $\Zi$ with $c=-1\pm 2i$, contrary to our hypothesis. We may thus assume that $f^2_r$ is irreducible. We proceed by induction on $n$, assuming $f^{n-1}_r$ is irreducible to prove that $f^n_r$ is irreducible. \\

We now follow this basic outline: Suppose $a_n$ is a square in $\Zi$. We first show $b_n^\pm \not\in \{0,\pm 1\}$ for all $c$ with $|c|> \sqrt{5}$. We also show $b_n^\pm \neq \pm 2i$ unless $c$ divides a power of $2$ (we handle this case separately); it will follow that both $b_n^+$, and $b_n^-$ have a prime divisor of odd multiplicity (i.e. are non-square).\\

Thus, suppose $a_n(c)$ is a square in $\Zi$ (if $a_n(c)$ is non-square in $\Zi$ then $f^n_r$ is irreducible by the proof of \cref{lemma2.2}). Then $b_n^\pm$ are $\Zi$-factors of $c^{2^{n-1}-1}$: $$a_n-a^2_{n-1} =c^{2^{n-1}-1}=(\sqrt{a_n}-a_{n-1})(\sqrt{a_n}+a_{n-1}) = i(a_{n-1}-\sqrt{a_n})i(a_{n-1}+\sqrt{a_n})=b_n^-b_n^+.$$

Also, note that $$\gcd(b_n^-, b_n^+)=\gcd(i(a_{n-1}-\sqrt{a_n}),i(a_{n-1}+\sqrt{a_n})) \mid \gcd(2ia_{n-1}, 2i\sqrt{a_n})=2$$ since $\gcd(a_n, a_{n-1})=\gcd(c^m, a_{n-1})=1$. When $c\equiv 1\bmod 2$, $c$ is coprime to $2$, so $\gcd(b_n^-, b_n^+)=1$.\\ % however, we will show $b_n^\pm$ is not a square in $\Zi$ for all $n\ge 3$ in this case while 

We first show that we may assume $b_n^{\pm}(c) \not\in \{0,\pm 1\}$. To show $b_n^\pm \neq0$, suppose $b_n^\pm =0$. Then $\sqrt{a_n} =\pm a_{n-1}$, so $a_n=a^2_{n-1} +c^{2^{n-1}-1}= a^2_{n-1}$. Since $c\neq 0$, this is impossible.\\

If $b_n^\pm =\pm 1$, then $a_{n-1}\pm \sqrt{a_n} = \pm i$. Then $a_n= -1 \mp 2ia_{n-1} +a^2_{n-1}$, so $$a^2_{n-1} +c^{2^{n-1}-1}= -1 \mp 2i a_{n-1} +a^2_{n-1}\implies c^{2^{n-1}-1} = -1 \pm 2i a_{n-1}(c).$$ Since $a_{n-1} \equiv 1 \bmod c$, we must then have $0 \equiv -1\pm 2i \bmod c$, so $c=\mu \cdot (-1\pm 2i)$ for some $\mu \in \Zi^\times$ since $-1\pm 2i$ is prime in $\Zi$. However, $c= \pm2+ i$ and $c=1+2i$ are found in \Cref{tab:2a} as pairs $(c \bmod z,z) = (i, 2)$ and $(1+2i, 4)$ respectively, so $a_n$ is non-square for all $n\ge 2$ for these $c$. Thus $f_{1/(\pm 2+i)}$, and $f_{1/(1+2i)}$ are stable by \cref{lemma2.2}. Note $f_r$ for $c=1-2i$ has the same behavior by \cref{prop1}. Thus we may assume $b_n^\pm \not\in \{0,\pm 1\}$.\\ %When $c= -1\pm 2i$, we have $\{a_n \bmod 3\}_{n\ge 1} \equiv \{1, \mp i, 1\mp i, 1\mp i, ...\}$ and $1\pm i$ is non-square mod 3, hence $a_n(-1\pm 2i)$ and thus $b_n^\pm (-1\pm 2i)$ is non-square for all $n\ge 3$.\\
%Such $c$ has $v_{1+i}(c)=0$, contrary to our hypothesis that $v_{1+i}(c)$ odd since $c\equiv_2 0$. However, $c\in \pm \{1\pm 2i, 2\pm i\}$ has $c\not\equiv 0 \bmod 2$, so \ref{itm:n0mod2} yields $b_n$ is non-square for all $n\ge 3$.\\

First suppose $c\equiv 1 \bmod 2$. We've shown that we may assume $b_n^\pm \not\in \{0,\pm1\}$. If say $b_n^+ \in \{-i, i\}$, then $b_n^+$ is non-square in $\Zi$, with $b_n^+ \equiv b_n^- \equiv i \bmod 2$, so $b_n^\pm$ is non-square in $\Zi$.\\

Now suppose $|b_n^\pm|>1$. Then $b_n^\pm$ has a prime divisor $\pi$ with $\pi\nmid b_n^\mp$, since $\gcd(b_n^-, b_n^+)=1$.  This prime divisor is of odd multiplicity since $v_\pi(b_n^\pm) = v_\pi(c^m)= m\cdot v_\pi(c)$ and both $m$ and $v_\pi(c)$ are odd by hypothesis. Thus $b_n^\pm$ is non-square.\\

%By hypothesis we have $v_{1+i}(c)\ge 3$, since if $v_{1+i}(c)=1$ then $c\equiv 1+i \bmod 2$ from \Cref{tab:2a}. Again, if $a_n\in \Zi^2$, then $b_n^\pm$ are $\Zi$-factors of $c^{2^{n-1}-1}$: $$a_n-a^2_{n-1} =c^{2^{n-1}-1}=b_n^-b_n^+.$$
%with $\gcd(b_n^-, b_n^+)\mid 2$. Again, since $\gcd(b_n^+, b_n^-) \mid 2$ then for a prime $\pi \neq 1+i$ we have either $\pi\mid b_n^+$ or $\pi \mid b_n^-$ but not both.\\

Now suppose $c\equiv 0\bmod 2$. We now show $b_n^\pm(c) \neq \pm 2i$ unless $c\mid 8$ or $c\mid 2^{n-1}$. The rest of the proof will be analogous to the $c\equiv 1\bmod 2 $ case above.\\ 

Suppose $b_n^\pm(c) = \pm2i$. Then $$a_{n-1}\pm \sqrt{a_n}= \pm2 \implies a_n= 4\mp 4a_{n-1} + a^2_{n-1}.$$ %so $\pm \sqrt{a_n}= \pm2 - a_{n-1}$, 

Therefore $c^{2^{n-1}-1}+a^2_{n-1} =4(1\pm a_{n-1})+a^2_{n-1}$, so $$c^{2^{n-1}-1}=4(1\pm a_{n-1}).$$

Since $c^{2^{n-1}-1}=4(1\pm a_{n-1})$, first suppose $c^{2^{n-1}-1}=4(a_{n-1}+1)$. Since $a_{n-1} \equiv 1 \bmod c$, then $4( a_{n-1}+1)\equiv 8 \equiv c^{2^{n-1}-1} \equiv 0 \bmod c$, so $c\mid 8=i(1+i)^6$.\\

Otherwise, we have $c^{2^{n-1}-1}=4(1-a_{n-1})$. Note that the constant term of $\frac{a_{n-1}(c)-1}{c}$ is $2^{n-3}$ for $n\ge 3$ by induction: when $n=3$, we have $\frac{a_2(c)-1}{c} = \frac{(c+1)-1}{c}=1$. Now suppose $k\ge 3$, and $\frac{a_{k-1}(c)-1}{c}$ has a constant term of $2^{k-3}$ to show $\frac{a_k(c)-1}{c}$ has constant term $2^{k-2}$. In particular, by the hypothesis we know $$a_{k-1}(c)=p_{k-1}(c) +2^{k-3} c + 1 $$ for some polynomial $p_{k-1}(c)$ divisible by $c^2$. Then 
    \begin{align*}
        a_k(c)=c^{2^{k-1}-1}+a^2_{k-1}(c)&=c^{2^{k-1}-1}+(p_{k-1}(c) +2^{k-3} c + 1)^2\\
        &= c^{2^{k-1}-1} + p^2_{k-1} + 2p_{k-1} \cdot 2^{k-3}c + 2p_{k-1} +(2^{k-3}c)^2+2\cdot 2^{k-3}c + 1\\
        &= p_k(c) + 2^{k-2} c +1
    \end{align*}

where $p_k$ is a polynomial divisible by $c^2$ which was the claim to be shown. Then $$0\equiv c^{2^{n-1}-2} \equiv 4\frac{1-a_{n-1}(c)}{c} \equiv 4 \cdot -2^{n-3} \equiv -2^{n-1} \bmod c$$ since the constant term of $\frac{a_{n-1}(c)- 1}{c}$ is $2^{n-3}$. Therefore $c\mid 2^{n-1}$ when $c^{2^{n-1}-1}=4(1-a_{n-1})$.\\

We now show that if $c$ divides a power of 2 (with $v_{1+i}(c)$ odd), then $b_n^\pm(c)$ is not a square in $\Zi$.
\begin{lemma}\label[lemma]{1+i}
If $c$ divides a power of 2 with $v_{1+i}(c)$ odd, then $b_n^\pm (c)$ is not a square in $\Zi$.
\end{lemma}
\begin{proof}

 If $v_{1+i}(c)$ is odd but $v_\pi(c)=0$ for all primes $\pi\neq 1+i$, then $c$ or $\bar{c}$ is of the form $\pm 2^k(1+i)$, both non-square so $f_r$ is irreducible. By \cref{prop1}, we assume $c=\pm 2^k(1+i)$.
 
 %$f_r^n$ has the same number of factors over $\Qi$ as $f_{\bar{r}}^n$, so from now on 
 
 %Note that by the remark following \cref{lem3.2}, when $c\not\in \R\cup i\R$ we have $b_n^\pm(\bar{c})=i(a_{n-1}(\bar{c})\pm\sqrt{a_n(\bar{c})})= \overline{-b_n^\pm(c)}$, so $b_n^\pm(\bar{c})$ is square in $\Zi$ if and only if $b_n^\pm(c)$ is. From now on we assume $c=\pm 2^k(1+i)$.\\

%It is easily shown that $c+1=\pm2^k+1\pm(2^k)i$ is non-square when $c\neq -4-4i$.

%if $$\pm2^k+1\pm(2^k)i=(x+iy)^2= x^2-y^2+2xyi,$$ then $2xy=2^k$, so $x$ or $y$ is $\pm1$ or else $x^2-y^2$ is even. If  $x=\pm 2^{k-1}$, $y=\pm 1$, then $$x^2-y^2=2^{2k-2} - 1 \neq \pm2^k+1=\Real(c+1)$$ for any integer $k\ge 1$. 
    
%     However, if $x=\pm1$ and $y=\pm 2^{k-1}$, then $$x^2-y^2=1-2^{2k-2} \neq \pm2^k +1=\Real(c+1)$$ unless $c=-2^k(1+i)$ and $k=2$. Then $c=-4-4i$ where $a_2=c+1=-3-4i=(-1+2i)^2$ is square.

%However, $-4-4i$ is not of the form $-m^2(m^2\pm2i)$, so $f^2$ is still irreducible since $f$ is irreducible by \cref{prop2.1}.\\
    % Note $a_2=c+1=-3-4i=(1-2i)^2$. We show that $b_2^\pm (-4-4i)$ is non-square: $$b_2^\pm(-4-4i)= i(a_1\pm \sqrt{a_2})=i(1\pm \sqrt{-3-4i})=\begin{cases}
    % i(1+1-2i)= 2+2i \\
    % i(1-(1-2i))=-2
    % \end{cases}\in \Zi\backslash\Zi^2.$$
    
    It is readily shown that $c+1=\pm2^k+1\pm(2^k)i$ is non-square when $c\neq -4-4i$. Note that $a_n(-4-4i)\equiv -1-i \bmod 3$. If $c\equiv -1-i\bmod 3$, we obtain the list $$\{a_n\}_{n\ge 1}\equiv \{1, -i, 1+i, 1+i, ...\}\bmod 3$$ and note that $1+i$ is non-square mod 3. Thus $a_n$ is non-square for $n\ge 3$, so $b_n^{\pm}$ is non-square for all $n\ge 3$. \\%It follows that $b_n^\pm(-4-4i)$ is non-square for all $n\ge 2$.  \\
    
    Now we may assume $a_2=c+1$ is non-square. Note $c\equiv 0\bmod 2$ implies we have $c+1\equiv 1 \bmod 2$. Since $a_2$ is non-square, it is not of the form $i\alpha^2$ since $i\alpha^2\in \{ 0, i\}\bmod 2$, hence $a_2$ has a prime divisor of odd multiplicity.\\
    
    Also note that $c=\pm2^k(1+i)\equiv \pm(1+i)\bmod 3$. When $c\equiv -1-i\bmod 3$, we've shown above that $b_n^\pm(c)$ is not a square in $\Zi$ for all $n\ge 3$.\\ %Since we may assume $c+1$ is non-square, we have $b_n^\pm(c)$ is not a square for all $n\ge 2$. \\
    
   From here on assume $c\equiv 1+i\bmod 3$. We show that $b_n^\pm \neq \pm 2i$; it will follow that $b_n^\pm$ is non-square. Suppose for now that $b_n^\pm \neq \pm 2i$ to show this implication. Recall that $a_n=a_{n-1}^2+c^m$ for $m=2^{n-1}-1$. Note $\gcd(b_n^-, b_n^+) \mid 2$ and $c^m=\pm i^{km}(1+i)^{(2k+1)m}=b_n^-b_n^+$  with $b_n^\pm\notin \{\pm 1, \pm 2i\}$. Suppose without loss of generality that $v_{1+i}(b_n^+)=2$. Since $b_n^+\neq \pm2i$, we then must have $b_n^+=\pm2$ and $v_{1+i} (b_n^-)$ odd (since $m$ and $2k+1$ are odd), so both $b_n^\pm$ are non-square.\\

If $v_{1+i} (b_n^+) =1$, then $b_n^+\equiv b_n^- \equiv 1+i \bmod 2$, so both $b_n^\pm$ are non-square.\\

Note that $v_{1+i} (b_n^+)=0$ is a contradiction since $b_n^-\equiv b_n^+ \equiv 1 \bmod 1+i$ while $$b_n^-\cdot b_n^+=c^m\equiv 0 \bmod 1+i.$$

In all cases, $b_n^\pm$ is non-square.\\

Recall that if $b_n^\pm = \pm 2i$, we have $c^m= 4(1\pm a_{n-1}) \implies c\mid 2^{n-1}$ or 8. We first show that the $c\mid 8$ cases pose no obstruction, obtained when $c^m=4(1+ a_{n-1})$.

%When $v_{1+i}(c)$ odd, we saw that we may assume $c$ is of the form $\pm 2^k(1+i)$, so $c$ is $\pm(1+i)\mod 3$. We also assume $c\equiv 1+i\bmod 3$.\\

% $i(a_{n-1}\pm \sqrt{a_n})=\pm2i$, so $a_{n-1}\pm \sqrt{a_n}=\pm 2$. Then $a_n= a^2_{n-1}\pm 4a_{n-1} +4$, and $a_{n-1}^2+c^m=a_{n-1}^2+4(1\pm a_{n-1})$. Thus  \\

If $c=\pm 2^k(1+i)$ divides 8, then $k\in\{1,2\}$ since by assumption $c\equiv 0 \mod 2$. The values of $c$ also satisfying $c\equiv 1+i\bmod 3$ are $c\in\{-2-2i,4+4i\}$. When $c= -2-2i$, we have that $c+1=-1-2i$ is non-square in $\Zi$. We find from \Cref{tab:2b} that the modulus $5+i$ ensures $a_n$ is non-square for $n\ge 2$. When $c=4+4i$, we find from \Cref{tab:2a} the modulus $3+i$. Thus $b_n^\pm(c)$ is non-square for all non-square $c\mid 8$ and $n\ge 2$. \\

%When $k=1$, $c=\pm(2+2i)$. We have already considered $c= 2+2i \equiv -1-i \bmod 3$ above. %$b_n^\pm(c)$ is non-square for all $n\ge 2$ in this case, since $a_2=c+1=3+2i$ is non-square.

When $c^m= 4(1-a_{n-1})$, we obtained $c\mid 2^{n-1}$. We now show that $b_n^\pm =\pm 2i$ and $c\equiv 1+i\mod 3$ leads to a contradiction. %Recall $c\equiv \pm (1+i)\bmod 3$ when $c\mid 2^{n-1}$ and $v_{1+i}(c)$ is odd.  
Now, we have 

\begin{equation}\label{eqcase}
   c^m=4(1-a_{n-1})\implies c^{2^{n-1}-1}=4(1-a_{n-2}^2-c^{2^{n-2}-1})
\end{equation}

Note we may assume $n>4$ since $a_3(c)$ is non-square by \cref{prop3.3} and since $a_4$ is non-square when $a_2$ is non-square, implying that $b_j^\pm$ is non-square in $\Zi$ for $j\in \{2,3,4\} \cup 2\Z^+$.\\

Note that $$\{c^k \bmod 3\}_{k\ge 1}\equiv \{(1+i)^k\bmod 3\}_{k\ge 1}\equiv \{1+i, -i, 1-i, 2, -1-i, i, -1+i, 1,...\}$$
so $1+i$ and thus $c$ has multiplicative order 8 mod 3. Since $8\mid 2^k$ when $k\ge 3$, we have $$c^{2^{n-1}-1} \equiv c^{-1}\equiv -1+i \mod 3.$$ Since $n>4$, we have $2^{n-1}-1\equiv 2^{n-2}-1 \equiv -1 \bmod 8$ and we obtain the following from \cref{eqcase}:
$$-1+i \equiv 1-a_{n-2}^2-( -1+i) \mod 3$$

We obtain $a_{n-2}^2\equiv -2i \bmod 3$, implying $a_{n-2}\equiv \pm (1-i)\bmod 3$. Since 
    $$\{a_n(c) \bmod 3\}_{n\ge1}\equiv \{1, -1+i, 1, i, 1+i, -1, i, 1+i, -1, ...\},$$
we must have $n-2=2$ in this case, so $n=4$. Since $n>4$ by assumption, this is a contradiction. 

%In \ref{itm:b}, we obtain $a_{n-2}^2\equiv 1 \bmod 3$, implying $a_{n-2}\equiv \pm 1 \bmod 3$. Again, from \cref{eqan}, we must have have $n-2\in \{1\}\cup\{3k\}_{n\ge 1}$ hence $n\in \{3\}\cup\{3k+2\}_{n\ge 1}$. However, $a_n$ is non-square in all such cases: we know from \cref{prop3.3} that $a_3(c)$ is non-square when $c\equiv 1+i \bmod 3$, and from \cref{eqan} that $a_{3k+2}\equiv 1+i\bmod 3$ is non-square mod 3. Again, $b_n^\pm \neq \pm 2i$. 

%Thus $b_n^\pm \neq 2i$; we've shown that this implies $b_n^\pm$ is non-square. %Since $a_2$ is also non-square by assumption, we've shown that $b_n^\pm$ is non-square for all $n\ge 3$.\\

%To prove \ref{itm:podd}, suppose $p\neq 1+i$ divides $c$ with $1+i\neq p$. Note that by \ref{itm:n0mod2} we may assume $2\mid c$, hence $1+i\mid b_n^\pm(c)$ since $c^m=b_n^+\cdot b_n^-$ for $m=2^{n-1}-1$. Since $b_n^-\equiv b_n^+ \bmod 2$, we must have either $b_n^\pm \equiv_2 0$ or $b_n^\pm \equiv 1+i \bmod 2$. If $b_n^\pm \equiv 1+i\bmod 2$, we're done, so suppose $b_n^\pm \equiv_2 0$.\\

%Now, if $b_n^\pm(c)= 2iq^2 \in \Zi^2$, then $\pm \sqrt{a_n}=2q^2-a_{n-1}$, similar to \ref{itm:odd}. Now $a_n=4q^4-4q^2a_{n-1} + a_{n-1}^2$, so $c^m=4q^4-4q^2a_{n-1}(c)=4q^2(q^2-a_{n-1})$. Thus 

Thus $b_n^\pm \neq 2i$; we've shown that this implies $b_n^\pm$ is non-square. Note that since $c$ is non-square, $f_r$ is irreducible. Since $\{b_n^+(c),b_n^-(c)\}$ contains no squares in $\Zi$ for each $n\ge 3$, it follows that  $f^n_{1/c}(x)$ is irreducible for all $n\ge 1$ by \cref{lem3.2}.
\end{proof}
\end{proof}
%end 3case
\begin{remark}
    One may apply \cref{thm3.6} to $c=-89$, a case with $a_3$ square in $\Zi$ by \cref{prop3.3}. Since $-89\equiv 3 \bmod4$ is a Gaussian prime, \cref{itm:odd} of this proposition applies (\cref{itm:1mod2} also applies since $2\nmid c$ and $16\nmid c+1=-88\notin \Zi^2$), hence we have $f^n_{-1/89}(x)$ is irreducible for all $n\ge 1$ by \cref{lem3.2}.
\end{remark}

\section{A useful factorization of $c$}
We first note the following extension of \cite[Lemma 4.1]{evstb} to the Gaussian integers (where $N:\mathbb{C}\to \R$ is the field norm $N_{\C/\R}(a+ib)=a^2+b^2$):  
\begin{lemma} \label[lemma]{lemma4.1} 
    Let $c\in \Zi$ with $|c|>1$, and $n\ge 2$ such that $a_n=a_n(c)$ is a square. Then we may write $c=uv$ with $\gcd(u,v)=1$, such that $$dv^{2^{n-1}-1}-\frac{u^{2^{n-1}-1}}{d}=2a_{n-1}(uv)$$ 
     where $d\in \{1,1+i, 2\}$ (if $n>2$, then $d\neq 1+i$). If in addition $\{b_n^+(c), b_n^-(c)\}$ contains a square, then $\{N(u),N(v)\}$ contains a square.
\end{lemma}
\begin{proof}
     This is the proof of \cite[Lemma 4.1]{evstb} with minor modifications. For a sketch of the proof: when $a_n(c)$ is a square, say $s^2$, then $c^m=(s+a_{n-1}(c))(s-a_{n-1}(c))$ for some $s\in \Zi$. Define $d:=\gcd(s+a_{n-1}(c),s-a_{n-1}(c))$, $t_+:=s+a_{n-1}(c)$, and $t_-:=s-a_{n-1}(c)$, so that $t_+t_-=c^m$ implies $c=uv$ with $t_+ =dv^m$ and $t_-=u^m/d$, or $t_+=u^m/d$ and $t_-=dv^m$.

     %note $d\mid 2$ as a non-unit proper divisor is impossible when $n>2$; since 2 has a proper divisor, 2 either ramifies (when $2\sim \pi^2$) or splits ($2\sim \pi \bar{\pi}$), so if $\pi$ is a prime divisor of $2$ we have $d\sim \pi$. Without loss of generality, $$\sigma+a_{n-1}(c)\equiv \sigma-a_{n-1}(c) \equiv d \bmod 2.$$  We thus have $v_{\pi}(\sigma \pm a_{n-1}(c))=1$ . Then, $$ m\cdot v_{\pi}(c) =v_{\pi}(c^m)= v_{\pi}(\sigma+a_{n-1}(c)) + v_{\pi} (\sigma-a_{n-1}(c))=2,$$ so $m$ odd implies $m=2^{n-1}-1$ divides $2$, hence $n=2$ (and $v_{\pi}(c)=2$).\\

     Note that $d=1+i$ is impossible when $n>2$: we have $$s+a_{n-1}(c)\equiv s-a_{n-1}(c) \equiv 1+i \bmod 2.$$ In this case, we have $v_{1+i}(s\pm a_{n-1}(c))=1$. However, $$v_{1+i}(c^m)= m\cdot v_{1+i}(c) = v_{1+i}(s+a_{n-1}(c)) + v_{1+i} (s-a_{n-1}(c))=2,$$ so $m$ odd implies $m=2^{n-1}-1=1$, hence $n=2$, and $v_{1+i}(c)=2$. \\

     Note $b_n^+=i\cdot t_+$ and $b_n^-=-i\cdot t_-$. If $b_n^+$ is a square in $\Zi$, then $|b_n^+|^2$ is a square in $\Zi$. If $d\mid 2$, we have $t_+= dv^m$ and $t_- = u^m/d$ or $t_+=u^m/d$ and $t_-=dv^m$. Then 
$$|b_n^+|^2 = \left|t_+\right|^2 \in \left\{\left|dv^m\right|^2, \left|\frac{u^m}{d}\right|^2\right\}=\left\{|d|^2(|v|^2)^m, \frac{(|u|^2)^m}{|d|^2}\right\}.$$ When $n>2$ we may assume $d\in \{1,2\}$, so one of $\{|u|^2, |v|^2\}$ is square in $\Zi$ since $m$ is odd. The case when $b_n^-$ is square is similar.\\

\end{proof}

We extend \cite[Definition 4.3]{evstb} over $\Zi$:
\begin{definition}\label[definition]{def4.3}
    Let $c\in \Zi$ with $|c|>1$. We set 
    
    $$Q(c):=\min\{N(v/u): u,v\in \Zi \text{ coprime with } |v|>|u| \text{ and } c=uv\}$$
    
\begin{center}
    $\tilde{Q}(c):=\min\{N(v/u): u,v\in \Zi \text{ coprime with } |v|>|u|, c=uv,$\\
    $\hspace{30mm} \text{ and } \{N(v), N(u)\} \text{ contains a square} \}$
\end{center}
    \center{and} 
     $$q(c):=\sqrt{Q(c)}, \quad \tilde{q}(c):=\sqrt{\tilde{Q}(c)}.$$
\end{definition}

If $Q(c)=N(v/u)$, then since $N(v)>N(u)$, $N(v)\ge N(u)+1$, thus $N(v/u)\ge 1+1/N(u)$. \\
Note $N(u)< N(\sqrt{c})$, since $N(u^2)<N(u)N(v)=N(c)$, so $$\tilde{Q}(c)\ge Q(c)=N\left(\frac{v}{u}\right)> 1+\frac{1}{N(\sqrt{c})}=1+\frac{1}{|c|} \hspace{5mm} \text{and} \hspace{5mm}\tilde{q}(c)\ge q(c) > \sqrt{1+\frac{1}{|c|}}.$$

We also make use of the following $\Zi$ extension of \cite[Definition 4.4]{evstb}:

%Definition 4.4
\begin{definition}\label{def4.4}
Let $c\in \Zi\backslash\{0,-1\}$ and $n\ge 2$. Define $\epsilon(n,c)$ so that $$\log\frac{\sqrt{a_n(c)}+a_{n-1}(c)}{\sqrt{a_n(c)}-a_{n-1}(c)} = \frac{\epsilon(n,c)}{\sqrt{c}}.$$
\end{definition}

Suppose $|g(z)|<1$. As in \cite{320169}, the ML inequality (or using Maclaurin series)  yields  
\begin{equation}\label{ineq}
    \left|\log\frac{1+g(z)}{1-g(z)}\right| = \left|2\int_0^{g(z)} \frac{d\gamma}{1-\gamma^2}\right| \le 2\int_0^{|g(z)|} \frac{dt}{1-t^2}= \log \frac{1+|g(z)|}{1-|g(z)|}.
\end{equation}
 This inequality prompts us to consider the ratio $\frac{a_{n-1}(c)}{\sqrt{a_n(c)}}$. We also consider the related ratio $\frac{a_n(c)}{c^{2^{n-1}-1}}$, which converges nicely when $c\in \R$, as shown by the following lemma. Assume $t\ge 4$ in $\R$, and let $F(t)=\frac{1}{2}(1-\sqrt{1-\frac{4}{t}})=\frac{2}{t(1+\sqrt{1-\frac{4}{t})}}$. Note $F(t)\le 1/2$, and that $F(t)$ decreases monotonically to 0. %From the second expression note that for large $t$, $F(t) \approx \frac{1}{t}$.\\

\begin{lemma}
\label[lemma]{lemma4.2} (Demark-Hindes-Jones-Misplon-Stoll-Stoneman, \cite[Lemma 4.2]{evstb})\\
    Let $t\ge 4$ in $\R$. Then the sequence $$\bar{a}_n(t)= \frac{a_n(t)}{t^{2^{n-1}-1}} \quad \text{for } n\ge 1$$
    satisfies $$1=\bar{a}_1(t)<\bar{a}_2(t)<... \hspace{5mm} \mathrm{ and } \hspace{5mm} \lim_{n\to \infty} \bar{a}_n(t) = tF(t).$$
\end{lemma}

We now seek some condition on $g(z):= \frac{a_{n-1}}{\sqrt{a_n}}(z)$ to guarantee $|\frac{a_{n-1}}{\sqrt{a_n}}(c)|<1$:
\begin{theorem}\label[theorem]{thmzineq1}
    Let $n\ge 2$. If $|c|\ge 5$, then $|\frac{a_{n-1}}{\sqrt{a_n}}(c)|< 1$.
\end{theorem}

\begin{proof}
   Let $m=2^{n-1}-1$. We consider the contrapositive. When $|\frac{a_{n-1}}{\sqrt{a_n}}(c)|\ge 1$, we have $|a_n(c)/a_{n-1}^2(c)|\le1$, thus $|a_n(c)|\le|a_{n-1}^2(c)|$. By the reverse triangle inequality and the relationship $a_n(c)=c^m+a_{n-1}^2(c)$, $$|c|^m-|a_{n-1}^2(c)|\le |a_{n-1}^2(c)+c^m|=|a_n(c)|\le |a_{n-1}^2(c)|$$ implies $|c|^m\le 2|a_{n-1}^2(c)|$. Therefore $|c|^m\le 2a_{n-1}^2(|c|)$ since $a_{n-1}(c)$ has non-negative coefficients, so $$1/2 \le \frac{a_{n-1}^2(|c|)}{|c|^m}=\frac{a_n(|c|)}{|c|^m} -1$$ 

since $a_n(|c|)= a_{n-1}^2(|c|)+|c|^m$. Thus $$3/2 \le \frac{a_n(|c|)}{|c|^m} =\bar{a}_n(|c|).$$

Since  $5\cdot F(5)\approx 1.38 <3/2$, we have that all $|c|=t \in[ 5,\infty)$ have $\bar{a}_n(t)< tF(t)< 3/2$ since $tF(t)$ decreases on $[4,\infty)$: $$\odv{}{t}(tF(t))= F(t)-t\cdot\left(\frac{1}{t^2\sqrt{1-4/t}}\right)=F(t)-\frac{1}{t\sqrt{1-4/t}}.$$\\
Since $F(t)= \frac{2}{t(1+\sqrt{1-\frac{4}{t})}}= \frac{2}{t+t\sqrt{1-4/t}}< \frac{2}{2t\sqrt{1-4/t}}=\frac{1}{t\sqrt{1-4/t}}$ for $t\in (4,\infty)$, we have $\odv{}{t}(tF(t))<0$ on $(4, \infty)$. Thus $tF(t)$ decreases to 1 on this interval.\\

Our hypothesis $|\frac{a_{n-1}}{\sqrt{a_n}}(c)|\ge 1$ is thus false on $[ 5,\infty)$. Thus if $|c|\ge 5$, then $|\frac{a_{n-1}}{\sqrt{a_n}}(c)|< 1$.
\end{proof}

\begin{remark}
    Note that $c\in \Zi$ with $|c|<5$ implies $|c|\le 2\sqrt{5}$, while $$2\sqrt{5}\cdot F(2\sqrt{5})\approx 1.51 >3/2,$$ so $5$ is the minimum modulus (for Gaussian integers) for which we can guarantee $|\frac{a_{n-1}}{\sqrt{a_n}}(c)|<1$ with this method.
\end{remark}

When $|c|\ge 5$, then $|\frac{a_{n-1}}{\sqrt{a_n}}(c)|<1$, hence$$
    |\epsilon(n,c)|=\left|\sqrt{c}\log\frac{\sqrt{a_n(c)}+a_{n-1}(c)}{\sqrt{a_n(c)}-a_{n-1}(c)}\right|= \left|\sqrt{c}\log\frac{1+\frac{a_{n-1}}{\sqrt{a_n}}(c)}{1-\frac{a_{n-1}}{\sqrt{a_n}}(c)}\right|\le\sqrt{|c|}\log\frac{1+|\frac{a_{n-1}}{\sqrt{a_n}}(c)|}{1-|\frac{a_{n-1}}{\sqrt{a_n}}(c)|}$$
% From the remarks following \cite[Lemma 4.2]{evstb} we know $\epsilon(n,|c|)\le \epsilon(|c|),$ so $$|\epsilon(n,c)| \le \epsilon(n,|c|) \le \epsilon(|c|)$$ for all $|c|> 5$.\\

Let $m=2^{n-1}-1$. Using the relation $a_n(c)-c^m=a_{n-1}^2(c)$, it follows that  $\frac{a_{n-1}}{\sqrt{a_n}}(c)=\sqrt{a_{n-1}^2/a_n}=\sqrt{1-c^m/a_n}$. Since $\bar{a}_n (c) = \frac{a_n(c)}{c^m}$ and $\bar{a}^2_{n-1}/c +1 =\bar{a}_n$ (as in the proof of \cref{lemma4.2}), we have that $$\left|\frac{a_{n-1}}{\sqrt{a_n}}(c)\right|= \left|\sqrt{1-\frac{1}{\bar{a}_n(c)}}\right|= \left|\sqrt{\frac{\bar{a}_n (c) -1 }{\bar{a}_n(c)}}\right|=\left|\sqrt{\frac{\bar{a}_{n-1}^2}{c\bar{a}_n}}\right|=\left|\frac{\bar{a}_{n-1}}{\sqrt{c\bar{a}_n}}\right|.$$

%When $|c|\ge5$, we have $|a_{n-1}(c)/\sqrt{a_n(c)}|<1$, so in particular $|\sqrt{(\bar{a}_n (c) -1 )/\bar{a}_n(c)}|<1$ and hence $$|\bar{a}_n (c) -1 |<|\bar{a}_n(c)|.$$\\

By the triangle inequality and \cref{lemma4.2}, $$|\bar{a}_n(c)| =\left| \frac{a_n(c)}{c^m}\right|\le \frac{a_n(|c|)}{|c|^m} =\bar{a}_n(|c|)\le |c|F(|c|).$$

Also, $|c|\ge 5$ implies the inequality $|a_{n-1}(c)/\sqrt{a_n(c)}|<1 \implies |a^2_{n-1}(c)|<|a_n(c)|$ by \cref{thmzineq1}. By the reverse triangle inequality, $|c^m| -|a_n(c)| \le |a_n(c)-c^m| =|a^2_{n-1}(c)|<|a_n(c)|$ implies $|c^m|< 2|a_n(c)|$, so $$\left|\frac{c^m}{a_n(c)}\right|=\left|\frac{1}{\bar{a}_n(c)}\right|<2 \implies \left|\frac{1}{\sqrt{\bar{a}_n}}\right|<\sqrt{2}.$$\\

By the previous inequalities, %$\frac{1}{\sqrt{2|c|}} <\left|\frac{\bar{a}_{n-1}}{\sqrt{c\bar{a}_n}}\right|$
$$\left|\frac{\bar{a}_{n-1}}{\sqrt{c\bar{a}_n}}\right| < \frac{|c|F(|c|)}{\sqrt{|c|}}\sqrt{2}=F(|c|)\sqrt{2|c|}.$$

Note $F(|c|)\sqrt{2|c|} = \frac{2\sqrt{2}}{\sqrt{|c|}(1+\sqrt{1-4/|c|})}$ decreases as a function of $|c|$ when $|c|\ge 4$. Since $|c|\ge 5$, we have $$F(|c|)\sqrt{2|c|}\le F(5)\sqrt{10}\approx 0.87.$$

%($F(x)\sqrt{2x}$ has derivative $-\frac{\sqrt{2}}{\sqrt{x}(x-4+\sqrt{(x-4)x})}$)

Since $\log \frac{1+x}{1-x}$ is an increasing function of $x$ on the interval $(-1,1)$, $$\sqrt{|c|}\log\frac{1+|a_{n-1}(c)/\sqrt{a_n(c)}|}{1-|a_{n-1}(c)/\sqrt{a_n(c)}|}= \sqrt{|c|}\log\frac{1+|\bar{a}_{n-1}|/|\sqrt{c\bar{a}_n}|}{1-|\bar{a}_{n-1}|/|\sqrt{c\bar{a}_n}|}< \sqrt{|c|}\log \frac{ 1+ F(|c|)\sqrt{2|c|}}{1-F(|c|)\sqrt{2|c|}}.$$

This bound leads to the following definition.
\begin{definition}\label[definition]{def4.5}
    Let $|c|\ge 5$, with $c\in \Zi$ and $n\ge 2$. Define $\Xi(|c|)$ as $$\Xi(|c|)= \sqrt{|c|}\log \frac{ 1+ F(|c|)\sqrt{2|c|}}{1-F(|c|)\sqrt{2|c|}}.$$
\end{definition}
We have shown that $|\epsilon(n,c)|< \Xi(|c|)$. It follows from the definition of $F(c)$ (and L'Hopital's rule) that $\Xi(|c|)$ decreases to $\sqrt{2}$ as $|c|\to \infty$. We also have that $\Xi(|c|)\le \Xi(5)\approx 6.05$.
\section{The proof of Theorem \ref{THM1.4}}
We now have the tools in place to prove the $\Zi$ analogue of \cite[Proposition 4.5]{evstb}: 
%Proposition 4.5
\begin{proposition}\label[proposition]{prop4.5}
     Suppose $|c|\ge  5$ and $n\ge5 $. If $c\in \Zi\backslash \Z^-$ and $$n\ge 1 + \log_2\left(1+\frac{|\epsilon(n,c)/\sqrt{c}|}{\log{q(c)}}+\frac{\log{4}}{\log{q(c)}}\right)$$
     then $a_n(c)$ is not a square.\\

If $c\in \Zi\backslash\Z^-$, and the weaker condition 
    $$n\ge 1 + \log_2\left(1+\frac{|\epsilon(n,c)/\sqrt{c}|}{\log{\tilde{q}(c)}}+\frac{\log{4}}{\log{\tilde{q}(c)}}\right)$$
holds, then $b_n^\pm(c)$ is non-square.
\end{proposition}
\begin{proof}
    In the following, we write $a_n$ for $a_n(c)$ since $c$ is fixed, and assume that $a_n$ is a square. Let $m=2^{n-1}-1.$ Since $n\ge2$, by \cref{lemma4.1} we may write $c=uv$ with coprime $u,v$ such that $$(v^m,u^m)= ( \frac{1}{d}(\pm \sqrt{a_n} + a_{n-1}), d(\pm \sqrt{a_n} - a_{n-1})) \quad \text{ for } d \in\{1,2\}.$$

    %$$(v^m,u^m)=(\frac{1}{1+i}(\pm\sqrt{a_n}+a_{n-1}),(1+i)(\pm\sqrt{a_n}-a_{n-1})) \text{ if } d=1+i$$

We first show $|u|\neq|v|$ in all cases. Note $|u|\neq|v|\iff |u^m|\neq |v^m|$. Let $\sqrt{a_n}=A+Bi$, $a_{n-1}=C+Di$.\\

When $d=1$, $N(u^m)=N(v^m)$ implies
\begin{equation*}
\begin{split}
N(\pm \sqrt{a_n}-a_{n-1})=N(\pm\sqrt{a_n} +a_{n+1})& \iff N(A+Bi-(C+Di))=N(A+Bi+(C+Di))\\
&\iff(A-C)^2+(B-D)^2 =(A+C)^2+(B+D)^2\\
&\iff -2AC-2BD=2AC+2BD\\
&\iff 0=4(AC+BD),
\end{split}
\end{equation*}
so the vectors $\overrightarrow{\sqrt{a_n}}=\begin{bmatrix}
    A \\
    B
\end{bmatrix}$ and $\vec{a}_{n-1}=\begin{bmatrix}
    C \\
    D
\end{bmatrix}$ are perpendicular.

Then, we must have $\overrightarrow{\sqrt{a_n}}=\lambda\begin{bmatrix}
    -D \\
    C
\end{bmatrix}$ for some real $\lambda$, i.e. $$\sqrt{a_n}= \lambda(-D+iC)=\lambda i(C+Di)=\lambda ia_{n-1}.$$

Then $A=-\lambda D$, $B=\lambda C$. Since $a_n$ is a square, $A, B\in \Z,$ so $\lambda$ is a rational number $s/t\in \Q$. To have $\sqrt{a_n}\in \Zi$ we must have $t\mid a_{n-1}$. Let $\beta=a_{n-1}/t \in \Zi$, so $\beta\mid a_{n-1}$. Then $$\sqrt{a_n}=(s/t)i a_{n-1}=si\beta.$$ However, $a_n$ is relatively prime to $a_{n-1}$ as shown in the proof of \cref{lemma4.1}. Since $\beta\mid \gcd(\sqrt{a_n}, a_{n-1})=1$, we have $\beta\in \{\pm 1, \pm i\}$. Thus $\sqrt{a_n}$ is either real or pure imaginary. If $\beta=\pm1 =a_{n-1}/t$, then $a_{n-1}\in \Z$, and thus $\sqrt{a_n}=\lambda ia_{n-1}\in i\Z$. The reverse is true for $\beta=\pm i$. In either case, $a_n\in \Z$.

Since $a_n=a^2_{n-1} +c^m$ and $a_n, a^2_{n-1}\in \Z$, then $c^m\in \Z$. Since $m=2^{n-1}-1$ is odd, we must have $c\in \Z$ (as in \cite{Michels2021}). This contradicts $a_{n-1}$ being pure imaginary, since $a_k(c)=a_{k-1}^2+c^{2^{k-1}-1}, a_1(c)=1$, and thus $a_k\in \Z$ whenever $c\in \Z$. Thus $\sqrt{a_n}$ is pure imaginary and $c\in \Z$ with $\beta=\pm 1$. The only way $\sqrt{a_n}$ is pure imaginary and $c\in \Z$ is when $a_n(c)$ is a negative integer. Then $c$ is a negative integer, since $a_n(c)>0$ when $c\in \Z^+$. In this case $A=D=0$, with $$N(\sqrt{a_n}-a_{n-1})=N(\sqrt{a_n}+a_{n-1}) \iff N(iB-C)=N(iB+C).$$ Since $c\in \Zi$ is not a negative integer by hypothesis, we must have $N(u)\neq N(v)$.\\

When $d=2$, $N(u^m)=N(v^m)$ implies $|\frac{\sqrt{a_n}+a_{n-1}}{\sqrt{a_n}-a_{n-1}}|=4$. Then $N(\frac{\sqrt{a_n}+a_{n-1}}{\sqrt{a_n}-a_{n-1}})=16$:
\begin{equation}\label{d2}
   N(\sqrt{a_n}+a_{n-1})=16\cdot N(\sqrt{a_n}-a_{n-1}).
\end{equation} 
Since $d=\gcd(\sqrt{a_n}+a_{n-1}, \sqrt{a_n}-a_{n-1}) = 2$, we must have $v_{1+i} (\sqrt{a_n}-a_{n-1}) = 2$, hence in $\Z$, $$v_2(N(\sqrt{a_n}-a_{n-1}))=2.$$ Now $v_2(N(\sqrt{a_n}+a_{n-1}))=6$ in $\Z$ by \cref{d2}. Since  $(\sqrt{a_n}+a_{n-1}) (\sqrt{a_n}-a_{n-1})=c^m$, $$8=v_2(N(\sqrt{a_n}+a_{n-1})N(\sqrt{a_n}-a_{n-1}))= v_2(N(c^m))= m\cdot v_{1+i}(c),$$ so $m\mid 8$. This is a contradiction; since $n>2$, $m=2^{n-1}-1$ is an odd number greater than 1. Thus $|v|\neq|u|$ in all cases.\\

Now, since $|v|\neq|u|$, \begin{equation*}
\begin{split}
m\log q(c) \le m|\log(|v/u|)| = |\log(|(v/u)^m|)|&=\left|\log\left|d^{\pm2}\frac{\sqrt{a_n}+a_{n-1}}{\sqrt{a_n}-a_{n-1}}\right|\right|\\
&= \left|\log|d^{\pm2}|+\log\left|\frac{\sqrt{a_n}+a_{n-1}}{\sqrt{a_n}-a_{n-1}}\right|\right| \\
% &\le \left|\log|d^{\pm2}|\right|+\left|\log\left|\frac{\sqrt{a_n}+a_{n-1}}{\sqrt{a_n}-a_{n-1}}\right|\right|\\
    & \le \log{4} + \left|\log \frac{\sqrt{a_n}+a_{n-1}}{\sqrt{a_n}-a_{n-1}}\right|\\
    &= \log{4} + |\epsilon(n,c)/\sqrt{c}|.
\end{split}
\end{equation*}

This is equivalent to the inequalities shown, since $m=2^{n-1}-1$.\\

Assuming that $b_n^\pm(c)$ is square in $\Zi$ yields $N(u)$ or $N(v)$ is square by \cref{lemma4.1}, hence the bound is valid for $\tilde{q}(c)$ in place of $q(c)$.\\
\end{proof} 

\begin{remark} Note that \cite[Proposition 4.5]{evstb} assumed $c$ is a non-square positive integer $c=uv$, so $u\neq \pm v$ and thus $|u|\neq|v|$ is automatic.
\end{remark}
We now handle the negative integer case. We'll need \cite[Proposition 3.1]{evstb} as a lemma:

\begin{lemma} (Demark-Hindes-Jones-Misplon-Stoll-Stoneman,\cite[Proposition 3.1]{evstb})\label[lemma]{lem:negative}
    If $c\in \Z^-$, then $a_n\in \Z^-$ for all $n\ge 2$.
\end{lemma}
% \begin{proof}
%     Let $r=1/c$ and $f_r(x)=x^2+r$, and consider the image of the interval $I=(-\sqrt{-r},0)$ under $f_r:\R\to \R$. We have $f_r(-\sqrt{-r})=0$ and $f_r(0)=r\in I$, so as $f_r$ is a continuous function with no critical points in $I$, it follows that $f_r(I)\subset I$. As $f_r(0)=r\in I$, inductively, $f_r^n(0)\in I$ for all $n\ge 1$. Hence $0>f_r^n(0)=a_n/c^{2^n}$, and hence $a_n<0$ for $n\ge 1$.
% \end{proof}
We can now prove that if $c\in \Z^-$ is non-square in $\Zi$, then $b_n^\pm$ is not a square in $\Zi$:
\begin{theorem}\label[theorem]{thm:negative}
    Suppose $c\in \Z^- \backslash \Zi^2$ is a negative integer. Then $f_r^n(x)$ is irreducible for all $n\ge1$.
\end{theorem}
\begin{proof}

    Since $c$ is a negative integer, $a_n(c)\in \Z^-$ by \cref{lem:negative}. Thus, if $a_n(c)$ is a square in $\Zi$, $\sqrt{a_n(c)}\in i\Z$. If $b_n^\pm (c)$ is a square in $\Zi$ for some $n>2$, then by the proof of \cref{prop4.5} we may write $c=uv$ with $\gcd(u,v)=1$ such that $\{N(u),N(v)\}$ contains a square, and with $$(v^m,u^m)=\left(\frac{1}{d}(\pm\sqrt{a_n}-a_{n-1}),d(\pm\sqrt{a_n}+a_{n-1})\right)=\left(\frac{iB-C}{d},d(iB+C)\right)\quad \text{ for } B, C\in \Z, \text{ with } d \in \{1,2\}.$$
%Without loss of generality, suppose $N(\frac{iB-C}{d})=N(u^m)$, with $N(v^m)=N(d(iB+C))$.
By multiplicativity of the norm $N(\cdot) = N_{\C /\R} (\cdot)$, we have $N(v^m)= d^4\cdot N(u^m)$. Thus one of $\{N(u),N(v)\}$ being a square forces the other to be a square. Note $|v^m|=d^2|u^m|$. Since $c^m=u^mv^m$, we have $|c^m|=|c|^m=|u^m||v^m|=d^2|u^m|^2=d^2N(u)^m$ is a square in $\Z$. Since $m$ is odd, we must have that $|c|=-c$ is a square in $\Z$ (so in particular, $c$ is a square in $\Zi$), contrary to our hypothesis. Thus $b_n^\pm(c)$ is not a square in $\Zi$ for all $n> 2$.\\

We now show $f_r^2(x)$ is irreducible to apply \cref{lem3.2}, it will follow that we have $f_r^n$ irreducible over $\Zi$ for all $n\ge 1$. If $c\neq \alpha^2(2i-\alpha^2) \in \Z^-\backslash \Zi^2$, then $f^2_r$ is irreducible by \cref{prop2.1}: suppose the opposite for a contradiction.

If $c= \alpha^2(2i-\alpha^2)$, then $c\in \{-1\pm 2i, 0\} \bmod 4.$ We may assume $c\equiv 0 \bmod 4$ since $c\in \Z^-$. Then $(1+i)\mid \alpha$, so $c=4\eta^2(\eta^2\pm 1)$ for some $\eta\in \Zi$. Since $c\in \Z$, we may assume $\eta\in \Z\cup i\Z$: if $\eta=A+Bi$, then $$0=\Imag(c)=\Imag(4\eta^2(\eta^2\pm1))=16A^3B-16AB^3\pm 8AB.$$ If both $A,B$ are non-zero, then $16A^2-16B^2\pm 8=0$, so $\pm8=16(B^2-A^2)$. This is a contradiction, so we may assume $A$ or $B=0$.

Thus $\eta \in \Z\cup i\Z$. If $|\eta|>1$ in $\Z$ or $i\Z$ we have $c=4\eta^2(\eta^2\pm 1) \in\Z^+$. If $|\eta|=1$, we have $c\in \{0, 8\}$. Thus $c=\alpha^2(2i-\alpha^2)\in \Z^-$ is a contradiction, so $c\neq \alpha^2(2i-\alpha^2)$.

\end{proof}

The lower bound on $n$ from \cref{prop4.5} leads to the following $\Zi$ analogue of \cite[Corollary 4.6]{evstb}:

\begin{corollary}\label{cor4.6}
     %(Demark-Hindes-Jones-Misplon-Stoll-Stoneman, )
     Let $f(x)=x^2+1/c$ for $c\in \Zi\backslash\Zi^2$ with $|c|\ge 5$.
\begin{enumerate}
    \item If $f^k$ is irreducible for $$k=1+\left\lfloor\log_2 \left( 1+\frac{\log{4}+\Xi(|c|)/\sqrt{|c|}}{\log{\sqrt{1+1/|c|}}}\right)\right\rfloor,$$
 then all $f^n$ are irreducible. \label{itm:mbound}
    \item If $a_2(c)=c+1$ and $a_p(c)$ are non-square for all prime numbers $p$ with $$5\le p\le 1+\left\lfloor\log_2 \left( 1+\frac{\log{4}+\Xi(|c|)/\sqrt{|c|}}{\log{\sqrt{1+1/|c|}}}\right)\right\rfloor,$$
    then all $f^n$ are irreducible. \label{itm:pbound}
    %damn\item If $a_2=c+1$ is non-square, $|c|>50$, and $\tilde{q}(c)\ge 1.15|c|^{-1/30}$, then all $f^n$ are irreducible.\label{itm:tilde}
\end{enumerate}
\end{corollary}
\begin{proof}
This is the proof \cite[Corollary 4.6]{evstb} using $|\epsilon(n,c)|\le \Xi(|c|)$ instead of $\epsilon(n,c)\le \epsilon(c)$.
\end{proof}

The $\Zi$ analogue of \cite[Proposition 5.4]{evstb} is again proved with minor modifications:

\begin{proposition}\label[proposition]{prop5.4}
%(Demark-Hindes-Jones-Misplon-Stoll-Stoneman, )
Let $c\in \Zi$ and set $f(x)=x^2+1/c$. If $N(c)\le 10^9$ and the second iterate $f^2$ is irreducible, then all iterates of $f$ are irreducible.
\end{proposition}
\begin{proof}
We consider the set of non-squares not of the form $\alpha^2(2i-\alpha^2)$: this is a necessary and sufficient condition for $f_r^2$ to be irreducible by \cref{prop2.1}. When $|c|< 5$, the only cases not handled by \cref{thm3.6} are $c \in \pm \{ 4i, 2+4i, 4+2i\}$, but $c=-4i$ has the same behavior as $4i$ by \cref{prop1}. \\

Note that when $c\in\{4i, \pm(4+2i), -2-4i\}$ we have that $c+1$ is prime in $\Zi$. Hence \Cref{tab:2b} applies to these $c$ with the following congruence pairs $(c,z)$: $(-i, 2+i)$ for $c= 4i\equiv -i\bmod 2+i$, $(\pm (1+i),3+i)$ for $c=\pm(4+2i)\equiv \pm (1+i) \bmod 3+i$, and $(-1+i,3+2i)$ for $c=-2-4i\equiv -1+i\bmod 3+2i$. Thus $a_n(c)$ is non-square in $\Zi$ for all $n\ge 2$ for these $c$.  \\

This leaves $c= 2+4i$, with $c+1=(2+i)^2$ a square. However $\{a_n\}_{n\ge3} =\{1-i, 1-i, ...\}\bmod 3$. Since $1-i$ is non-square mod 3, $a_n$ is non-square for all $n\ge 3$. Since $f^2$ is irreducible by hypothesis, $f^n$ is irreducible for all $n\ge 1$  by \cref{lem3.2}.\\

%When $c\in\{-2, 2+4i\}$ then $c+1\in \{-1,3+4i\}=\{i^2, (2+i)^2\}$. Thus $c+1$ is square in $\Zi$. .%If $c=2$, consider $\{a_n\}_{n\ge 1}\equiv \{1,-1,1+2i , 2-i, 1+2i,...\}\bmod 4+i$ and note that $1+2i, 2-i$ are non-squares mod $4+i$, so $a_n$ is non-square for all $n\ge 3$. Since $c=-2$ has $f_r^2$ irreducible by hypothesis, $f^n_r$ is irreducible for all $n\ge 1$ by \cref{lemma2.2}.\\

We now assume $|c|\ge5$. By \cref{thm3.6} it suffices to consider $2\mid c$ or $16\mid c+1$, with $\pm c$ in the first quadrant by Corollary \ref{prop1}. In either case, when $a_2(c)$, $a_p(c)$ are non-square for integer primes 

\begin{equation}
    \begin{split}
    5\le p&\le 1+\left\lfloor\log_2 \left( 1+\frac{\log{4}+\Xi(10^{4.5})/\sqrt{10^{4.5}}}{\log{\sqrt{1+1/10^{4.5}}}}\right)\right\rfloor = 17
\end{split}
\end{equation}

then all $f^n$ are irreducible by Corollary \ref{cor4.6}. This is a finite computation. \\
\end{proof}

 \section{Applications to the density of primes dividing orbits}

 \quad Let $K$ be a number field with ring of integers $\mathcal{O}_K$. Let $S\subset \mathcal{O}_K$ be a set of prime ideals, $\mathfrak{p}\in S$, and denote by $N(\mathfrak{p})=\#\{\mathcal{O}_K/\mathfrak{p}\mathcal{O}_K\}$ the absolute norm of the ideal $\mathfrak{p}$. We denote by $D(S)$ the Dirichlet density of the set $S$: $$D(S)= \lim_{s\to 1^+} \frac{\sum_{\mathfrak{p}\in S} N(\mathfrak{p})^{-s}}{\sum_{\mathfrak{p}} N(\mathfrak{p})^{-s}}$$
 
 Since $K$ is a number field, this density can be shown to be equivalent to the the natural density of $S$, denoted by $$\delta(S) = \lim_{B\to \infty} \frac{\#\{\mathfrak{p}\in S : N(\mathfrak{p})\le B\}}{\#\{\mathfrak{p}: N(\mathfrak{p})\le B\}}.$$
 
The following is an extension to $\Zi$ of \cite[Proposition 6.1]{evstb}, requiring minor modifications:
 \begin{theorem}\label[theorem]{thm6.1}
 %(Demark-Hindes-Jones-Misplon-Stoll-Stoneman, )
     Let $c,c+1\in \Zi\backslash\Zi^2$, $r=1/c$, and assume that $\{|b_n^+|^2, |b_n^-|^2\}$ contains no squares in $\Z$ for all $n\ge 3$. Then for any $t\in \Qi$, we have $$D(\{\pi \text{ prime in } \Zi: \pi \text{ divides } O_{f_r}(t)\})=0.$$
 \end{theorem}
 \begin{proof}
Let $H:=\Qi(\sqrt{r},\sqrt{\bar{r}})$, and note $H_1:=\Qi(\sqrt{r})$ (from the proof of \cref{lem3.2}) is a subfield. Now $f_r=(x+i\sqrt{r})(x-i\sqrt{r})$ over $H$ or $H_1$. Let $g_1=(x+i\sqrt{r})$, $g_2=(x-i\sqrt{r})$. We'll show $g_i(f^{n-1}(0))$ is non-square in $H$ for all $n\ge 2$. Note $g_i(f^{n-1}_r(0))=f_r^{n-1}(0)\pm i\sqrt{r}\in K_1,$ and that following the proof of \cref{lem3.2}, $f_r^{n-1}(0)\pm i\sqrt{r}$ is a square in $H_1$ if and only if $i(f_r^{n-1}(0)\pm \sqrt{f^n_r(0)})$ is a square in $\Qi,$ which is equivalent to $b_n^\pm=i(a_{n-1}\pm \sqrt{a_n})$ being a square in $\Zi$. By assumption, $b_n^\pm$ is non-square in $\Zi$ for any $n\ge 3$, so $g_i(f^{n-1}_r(0))$ is not a square in $H_1$. If it were square in $H$, then there exist $a,b\in H_1$ such that $g_i(f^{n-1}_r(0))=(a+b\sqrt{\bar{r}})^2$ with $b\neq 0$. Then $g_i(f^{n-1}_r(0))=f_r^{n-1}(0)\pm i\sqrt{r}$ can be written as
        $$a^2+\bar{r}b^2+2ab\sqrt{\bar{r}}= f_r^{n-1}(0)\pm i\sqrt{r}.$$
 In particular, $2ab=0$ otherwise $2ab\sqrt{\bar{r}}\not\in H_1$ can be written as $ f_r^{n-1}(0)\pm i\sqrt{r} -(a^2+\bar{r}b^2)\in H_1$, so $a=0$ since $b\neq 0$ by assumption. Thus $$f_r^{n-1}(0)\pm i\sqrt{r} = \bar{r}\cdot b^2\in \Qi$$
is a contradiction since  $f_r^{n-1}(0)\pm i\sqrt{r}\in H_1\backslash\Qi$. Thus $f_r^{n-1}(0)\pm i\sqrt{r}$ is not a square in $H$. Since $b\in H_1$, $b=x+yi\sqrt{r}$ for some $x,y\in \Qi$, so $f_r^{n-1}(0)\pm i\sqrt{r} = \bar{r}\cdot (x^2-ry^2+ 2xyi\sqrt{r})=\bar{r}(x^2-ry^2)+2xyi\bar{r}\sqrt{r}$, so $2xy\bar{r}=\pm 1$ implies $y=\frac{\pm 1}{2x\bar{r}}$. Therefore $$f_r^{n-1}(0) =\bar{r}(x^2-r(\frac{\pm 1}{2x\bar{r}})^2)\implies \bar{r}x^4-f_r^{n-1}(0) x^2- \frac{r}{4\bar{r}}=0.$$ Thus $x^2=\frac{f_r^{n-1}(0)\pm \sqrt{f^n_r(0)}}{2\bar{r}}=\frac{a_{n-1}(c)\pm \sqrt{a_n(c)}}{2c^{2^{n-1}}\bar{r}}$, so  $f_r^{n-1}(0)\pm i\sqrt{r}$ is a square in $H$ if and only if $\frac{f_r^{n-1}(0)\pm \sqrt{f^n_r(0)}}{2\bar{r}}$ is a square in $\Qi$.\\

%Note $\overline{x^2}=\overline{(\frac{a_{n-1}(c)\pm \sqrt{a_n(c)}}{2c^{2^{n-1}}\overline{r}})}$, so
Observe that
\begin{align*}
N_{\Qi/\Q}(x^2)=\overline{x^2} \cdot x^2 &= \overline{(\frac{a_{n-1}(c)\pm \sqrt{a_n(c)}}{2c^{2^{n-1}}\overline{r}})}\cdot\frac{(a_{n-1}(c)\pm \sqrt{a_n(c)})}{2c^{2^{n-1}}\bar{r}}\\
&= \frac{N_{\Qi/\Q}(a_{n-1}(c)\pm \sqrt{a_n(c)})}{4N_{\Qi/\Q}(c^{2^{n-1}})N_{\Qi/\Q}(r)}\\
&= \frac{N_{\Qi/\Q}(a_{n-1}(c)\pm \sqrt{a_n(c)})}{4N_{\Qi/\Q}(c^{2^{n-1}})N_{\Qi/\Q}(\frac{1}{c})}\\
 &= \frac{N_{\Qi/\Q}(a_{n-1}(c)\pm \sqrt{a_n(c)})}{4N_{\Qi/\Q}(c^{2^{n-1}-1})}\\
 &=\frac{1}{4N_{\Qi/\Q}(a_{n-1}(c)\mp \sqrt{a_n(c)})}. 
\end{align*}
By hypothesis, neither of $N_{\Qi/\Q}(b_n^\pm (c))=|b_n^\pm |^2 =N_{\Qi/\Q}(a_{n-1}(c)\pm \sqrt{a_n(c)}))$ are a square in $\Z$. This is a contradiction since $N_{\Qi/\Q}(x^2)$ is a square in $\Q$. Thus $g_i(f_r^{n-1}(0))$ is not a square in $H$. Note also that $- g_i(f(0))=-r\mp \sqrt{-r}$ is a square in $H\iff \frac{-r\mp \sqrt{r^2+r}}{2\bar{r}}=\frac{-1\mp \sqrt{c+1}}{2c\bar{r}}$ is a square in $\Qi$, but $c+1$ is non-square in $\Qi$ by hypothesis.
%Moreover, $g_i(f_r(0))= r\pm i\sqrt{r}$, which is a square in $K$ if and only if $(r\pm \sqrt{r^2+r})/2$ is a square in $\Qi$. Because $c+1$ is non-square in $\Qi$, it follows that $r^2+r=\frac{1}{c^2}+\frac{1}{c}=\frac{1+c}{c^2}$ is not a square in $\Qi$ either, proving that $g_i(f_r(0))$ is non-square in $\Qi$. \\

Thus we may apply part $(2)$ of \cite[Theorem 1.1]{orbits} twice to show 

\begin{equation}\label{pdens}
    0=\lim_{B\to \infty} \frac{\# \{ \mathfrak{p}\in S: N(\mathfrak{p})\le B\}}{\# \{ \mathfrak{p}: N(\mathfrak{p})\le B\}}
\end{equation}
where $N(\mathfrak{p})= \#(\mathcal{O}_H /\mathfrak{p}\mathcal{O}_H)$ is the absolute norm of the ideal $\mathfrak{p}$ and $S$ is the set of primes $\mathfrak{p}$ in the ring of integers $\mathcal{O}_H$ of $K$ that divide $g_i(f^{n-1}_r(t))$ for at least one $i\in \{1,2\}$ and at least one $n\ge 2$.\\

Note $H=\Qi(\sqrt{r},\sqrt{\bar{r}})$ is a Galois extension of $\Q$ of degree 8. Excluding ramified primes $\mathfrak{p}\in \mathcal{O}_H$ (this is a finite set of primes), $\mathfrak{p}\in \mathcal{O}_H$ either has absolute norm $p\in \Z$ (when $p$ splits in $\mathcal{O}_H$), or absolute norm $p^k$ for $2\le k\le 8$ otherwise. We may thus assume that $N(\mathfrak{p})=p^k$ for some $1\le k\le 8$.\\
 %This is because $K=\Qi(\sqrt{r})$ is a quadratic Galois extension of $\Qi$, and $N_{K/\Qi}(\mathfrak{q})= \mathfrak{p}^{[\mathcal{O}_K/\mathfrak{q}:\Zi/\mathfrak{p}]}$ when $\mathfrak{p}$ is the prime ideal lying below $\mathfrak{q}$.
 
 %If $\mathfrak{p}$ has norm $p$ then $p$ splits in $\mathcal{O}_K$. If $\mathfrak{p}$ has norm $p^2$ then $p$ is inert in $\mathcal{O}_K$.\\
 
 Note that for $k\ge 2$, $\#\{n\le B: n=p^k \text{ for some prime } p\}$ has asymptotic density zero relative to $\# \{n\le B: n=p  \text{ for some prime } p\}$ and so \cref{pdens} is equivalent to 

 \begin{equation}\label{pdens1}
     0=\lim_{B\to \infty} \frac{\# \{ \mathfrak{p}\in S: N(\mathfrak{p})=p\le B\}}{\# \{ \mathfrak{p}: N(\mathfrak{p})=p\le B\}}
 \end{equation}

Suppose $\mathfrak{p}\in S$, and say $\mathfrak{p}\mid g_i(f^{n-1}_r(t))$ for some $n\ge 2$. Then
$$N_{H/\Qi}(\mathfrak{p}) \mid N_{H/\Qi}(g_i(f_r^{n-1}(t)))=\prod_{\sigma\in \Gal(K/\Qi)} \sigma(g_i(f_r^{n-1}(t)))= (f_r^n(t))^2.$$

Let $(p)=\Z\cap \mathfrak{p
}\mathcal{O}_H$, $(\pi)=\Zi\cap \mathfrak{p}\mathcal{O}_H$ be the primes lying below $\mathfrak{p}$, and note $\pi \mid N_{H/\Qi}(\mathfrak{p})$. Note that $N(\mathfrak{p})=p$ if and only if $p$ splits in $\mathcal{O}_H$. This is true if and only if $\pi$ splits in $\mathcal{O}_H$, i.e. $N_{\Qi/\Q}(\pi)=p$ and $r$ is a quadratic residue mod $\pi$ and $\overline{\pi}$. Otherwise, $N(\mathfrak{p})=p^k$ for some $2\le k\le 8$. But $\pi\mid (f^n_r(t))^2$, so $\pi \mid f^n_r(t)$ and thus $\overline{\pi}\mid \overline{f^n_r(t)}=f^n_{\overline{r}}(\overline{t})$ with $$0\equiv f_r(f_r^{n-1}(t))\equiv (f_r^{n-1}(t))^2+r \bmod \pi$$ and $$0\equiv f_{\overline{r}}(f_{\overline{r}}^{n-1}(\overline{t}))\equiv (f_{\overline{r}}^{n-1}(\overline{t}))^2+\overline{r} \bmod \overline{\pi}.$$ Thus $r$ must be a quadratic residue mod $\pi$, and $\overline{r}$ is a quadratic residue mod $\overline{\pi}$ implies $r$ is one as well. %We must also have $-1$ a quadratic residue mod $p$: if $p\equiv 3\bmod4$, $p$ is inert in $\Zi$, hence $N(\mathfrak{p})=p^k$ for some $2\le k\le 4$. 
It follows that the numerator of \cref{pdens1} is $$4\#\{\pi \in \Zi : N(\pi)=p\le B \text{ and } \pi \text{ divides } O_f(t) \}.$$ The denominator is
$$4\#\{\pi \in \Zi: N(\pi)=p\le B \text{ and } r \text{ is a quadratic residue mod } \pi \text{ and } \overline{\pi}\}.$$
By Chebotarev's density theorem, the proportion of integer primes that split completely in a Galois extension of degree $n$ is $1/n$, thus the latter is asymptotic to $C\cdot \#\{\pi\in \Zi:N(\pi)=p\le B\}$ for some constant $C$. It follows that $$D(\{\pi \text{ prime in } \Zi: \pi \text{ divides } O_f(t)\})=0$$ as desired.
\end{proof}
%N_{\Qi/\Q}(b_n^\pm(c))==N_{\Qi/\Q}(a_{n-1}\pm \sqrt{a_n})
\begin{remark}
    The hypothesis that $|b_n^\pm(c)|^2$ are non-square is strictly weaker than $a_n(c)$ being non-square for $n\ge 2$. When $a_n(c)$ is non-square for all $n\ge 2$, the conclusion follows by (2) of \cite[Theorem 1.1]{orbits}.
\end{remark}

% We also can now complete the case over $\Q$:
% \begin{theorem}
%     Let $b_n^\pm(c)$ be defined as in \ref{bn}. If $c\in \Z_{>1}$ with $c\equiv 1\mod 2$, then $b_n^\pm(c)$ is not a square in $\Zi$  for all $n\ge3$.
% \end{theorem}

% \begin{proof}
%     When $c\in \Z_{>1}$, $a_n(c) \in \Z_{>1}$, so $a_n(c)\in \Zi^2\implies a_n(c)\in \Z^2$. Suppose $a_n(c)$ is a square, or else $b_n^\pm$ is not a square in $\Zi$. In particular, $b_n^\pm (c)=i(a_{n-1}(c)\pm \sqrt{a_n}(c))\in i\Z$. The only squares in $i\Z$, i.e. $\Zi^2\cap i\Z=\{\pm 2x^2 i:x\in \Z\}$: note $(a+ib)^2= a^2-b^2 + 2abi$ so $a=\pm b$ if there is no imaginary part. Since $c^{2^{n-1}-1}= b_n^- b_n^+$, we have a contradiction with $1\equiv 0 \mod 2$.
    
% \end{proof}
\appendix

\bibliographystyle{acm}
\bibliography{sample}

\end{document}